\documentclass[12pt]{amsart}
\usepackage{amsmath} 
\usepackage{amsfonts}
\usepackage{amssymb}
\usepackage{amsthm}
\usepackage{enumerate}
\usepackage{tikz}
\usepackage{color}
\usepackage[colorlinks,citecolor=blue,backref=page]{hyperref}
\usepackage{latexsym}
\usepackage{mathtools,accents}
\usepackage{mathrsfs}
\usepackage{aliascnt}
\usepackage{braket}
\usepackage{bm}
\usepackage[margin=1in,marginpar=0.75in]{geometry}

\newtheorem{thm}{Theorem}[section]

\newtheorem{lemma}[thm]{Lemma}

\newtheorem{coro}[thm]{Corollary}

\newtheorem{prop}[thm]{Proposition}
\newtheorem{remark}[thm]{Remark}

\newtheorem{definition}[thm]{Definition}

\newtheorem{example}[thm]{Example}
\numberwithin{equation}{section}

\newcommand{\pr}{\partial}
\newcommand{\ta}{\tilde{\alpha}}

\newcommand{\veps}{\varepsilon}
\newcommand{\lm}[2]{\lim\limits_{#1\to #2}}
\newcommand{\ovr}[1]{\overline{#1}}   

\newcommand{\be}{\begin{equation}}
\newcommand{\ee}{\end{equation}}
\newcommand{\bee}{\begin{equation*}}
\newcommand{\eee}{\end{equation*}}
   
\def\tr{\textnormal{tr}}
\def\dv{\textnormal{div}}

\def\grad{\nabla}

\def\dint{\displaystyle\int}
\def\R{\mathbb{R}}
\def\vol{\mathrm{vol}}
\def\dvol{\mathrm{dvol}}
\def\r{\rho}

\def\S{\Sigma}
\def\({\left(}
\def\){\right)}
\def\a{\alpha}
\def\b{\beta}

\def\H{\mathcal{H}}
\def\bH{\mathbf{H}}

\def\Hb{\mathbb{H}}
\def\s{\sigma}

\def\w{\omega}
\def\g{\gamma}
\def\V{\mathcal{V}}

\def\bS{\mathbb{S}}

\def\graph{\textnormal{graph}}

\def\m{\mathbf{m}}

\def\depth{\textnormal{depth}}
\def\G{\mathcal{G}}
\def\l{\ell}
\def\diam{\textnormal{diam}}
\def\GHto{\stackrel{\textnormal{GH}}{\longrightarrow}}
\newcommand{\rstr}{\:\mbox{\rule{0.1ex}{1.2ex}\rule{1.1ex}{0.1ex}}\:}
\def\lip{\textnormal{Lip}} 
\newcommand{\IFto}{\stackrel {\mathcal{F}}{\longrightarrow} }

\newcommand{\bdry}{\partial}
\def\set{\textrm{set}}
\def\spt{\textrm{spt}}
\newcommand{\mass}{{\mathbf M}}
\newcommand{\intcurr}{{\mathbf I}}

\newcommand{\IFZto}{\xrightarrow[Z]{\;\; \mathcal{F}\;\; }}

\newcommand{\di}{d^{\mathrm{intr}}}

\newcommand{\definedas}{\mathrel{\raise.095ex\hbox{\rm :}\mkern-5.2mu=}}
 \newcommand{\asdefined}{\mathrel{=\mkern-5.2mu\raise.095ex\hbox{\rm :}}}
\def\defeq{\definedas}

\allowdisplaybreaks

\title[IF stability of the PMT for AH graphs]{Intrinsic flat stability of the positive mass theorem for asymptotically hyperbolic graphical manifolds}
\author[{Cabrera Pacheco}]{Armando J. {Cabrera Pacheco}}
\address{Department of Mathematics, Universit\"at T\"ubingen, 72076 T\"{u}bingen, Germany}
\curraddr{T\"ubingen AI Center, Universit\"at T\"ubingen, 72076 T\"{u}bingen, Germany}
\email{a.cabrera@uni-tuebingen.de}

\author[{Graf}]{Melanie Graf}
\address{Department of Mathematics, Universit\"at T\"ubingen, 72076 T\"{u}bingen, Germany and Institute for Mathematics, Universit\"at Potsdam, 14476 Potsdam, Germany}
\curraddr{Department of Mathematics, Universit\"at Hamburg,  20146 Hamburg, Germany}
\email{melanie.graf@uni-hamburg.de}

\author[Perales]{Raquel Perales}
\address{CONACYT Research Fellow. Institute of Mathematics at the National Autonomous University of Mexico. Oaxaca}
\email{raquel.perales@im.unam.mx}

\begin{document}

\begin{abstract}
The rigidity of the Riemannian positive mass theorem for asymptotically hyperbolic manifolds states that the total mass of such a manifold is zero if and only if the manifold is isometric to the hyperbolic space. This leads to study the stability of this statement, that is, if the total mass of an asymptotically hyperbolic manifold is almost zero, is this manifold close to the hyperbolic space in any way? Motivated by the work of Huang, Lee and Sormani for asymptotically flat graphical manifolds with respect to intrinsic flat distance, 
we show the intrinsic flat stability of the positive mass theorem for a class of asymptotically hyperbolic graphical manifolds by adapting the positive answer to this question provided by Huang, Lee and the third named author.
\end{abstract}

\maketitle


\section{Introduction}\label{sec-intro}
\thispagestyle{empty}

In the context of mathematical relativity, asymptotically hyperbolic manifolds correspond to initial data sets for the Einstein equations with a negative cosmological constant $\Lambda$. An asymptotically hyperbolic manifold is, roughly speaking, a Riemannian manifold $(M^n,g)$  such that the metric $g$ approaches the metric of the $n$-dimensional hyperbolic space, $\Hb^n$, at infinity sufficiently fast.  Under appropriate decay conditions on $g$ there is a well defined notion of total mass of $(M^n,g)$ given by Chru\'sciel and Herzlich \cite{C-H}, and Wang \cite{Wang}.

From the constraint equations for the Einstein equations, it follows that if the initial data set is a time-symmetric, asymptotically hyperbolic manifold $(M^n,g)$, then 
with a suitable normalization of the cosmological constant,  the dominant energy condition reduces to a lower bound on the scalar curvature, $R(g) \geq -n(n-1)$.  The positive mass theorem then asserts that the mass of this type of manifolds is non-negative, and it is equal to 0 if and only if the manifold is isometric to $\Hb^n$. The history of the positive mass theorem and its different proofs is rich.  We refer the reader to a recent proof by Sakovich \cite{Sakovich}, which also contains a complete description of the history of this result and previous results. The rigidity part was established in general by Huang, Jang and Martin in \cite{H-J-M}.

From the rigidity statement of the positive mass theorem it is natural to ask whether a stability statement holds. The answer to this question is subtle, as can be seen in examples for the asymptotically flat setting given by Lee and Sormani in \cite{L-S}, showing that the answer is negative with respect to some usual topologies. Nonetheless, the stability of the positive mass theorem has been established in some cases.  In particular, the Sormani--Wenger intrinsic flat distance \cite{SW2} has shown to be an adequate notion of distance for this problem. In \cite{S-S} Sakovich and Sormani obtained a stability result for the positive mass theorem for complete rotationally symmetric asymptotically hyperbolic manifolds with respect to this distance. Huang and Lee \cite{H-L}   showed stability of the positive mass theorem with respect to the Federer--Fleming flat distance for a class of asymptotically flat graphical manifolds.   Subsequently Huang, Lee and Sormani showed stability of the positive mass theorem for a smaller subclass 
with respect to the intrinsic flat distance 
\cite{HLS}. While there were two gaps in \cite{HLS}, see \cite{HLS-Corr}, these are now completely filled in. The gap in \cite[Theorem 1.3]{HLS} was filled in by work of Del Nin and the third named author \cite{dNP}, and the gap in  \cite[Theorem 1.4]{HLS} by the work of Huang, Lee and the third named author  \cite{HLP}.  There is also a proof of  \cite[Theorem 1.3]{HLS}
using a different approach to \cite{HLS} in \cite{HLP} (c.f. \cite[Theorem 7.2]{A-P} for  a proof in the entire case).   

Following the work in \cite{H-L}, the first named author showed in  \cite{C} the stability of the positive mass theorem for a class of asymptotically hyperbolic graphical manifolds with respect to the Federer--Fleming flat distance.  In this work, starting from \cite{C} and following  \cite{HLP}, we establish a stability result  for a subclass of asymptotically hyperbolic graphs with respect to the intrinsic flat distance. We note that we cannot use the results in \cite{dNP} since they only ensure existence of intrinsic flat limits of the form $(B^{\R^n}(R), d_{\R^n}, [[ B^{\R^n}(R)]])$.  

For asymptotically hyperbolic manifolds stability has been established in some other cases (see \cite{S-S,AllenAH}).  The techniques in  \cite{H-L} have been also successfully applied to obtain a stability result of the Brown--York mass by Alaee, McCormick and the first named author in \cite{ACM}, and the techniques in \cite{H-L,HLS,A-P, HLP}  have also been applied to obtain flat and intrinsic flat stability results for tori with almost non-negative scalar curvature by the first and third named authors \cite{CKP}.

We will follow the definition of asymptotically hyperbolic graphs and the adaptation of the total mass for asymptotically hyperbolic manifolds to asymptotically hyperbolic graphs by  Dahl, Gicquaud and Sakovich \cite{D-G-S}.  We write $\Hb^{n+1}$ as the warped product $\Hb^n \times_{V} \R$ with metric $\bar b =  b + V^2 ds^2$, where $b$ is the metric of $\Hb^n$
which we write in coordinates defined on $\R^n=[0,\infty) \times \bS^{n-1}$ and $V(r) = \cosh(r)$ where $r$ represents the radial coordinate. We denote the open ball in $\Hb^{n}$ of radius $\r$ around the origin by $B^b(\r)$.   We start by defining the class of asymptotically hyperbolic graphs that where studied in  \cite{C} and which will be the basis for the subclass we will consider here.

\begin{definition}\label{def-Gn}
For $n  \geq 3$, define $\mathcal{G}_n$ to be the space of  graphs 
of functions, $\graph(f) \subset  \Hb^{n+1}$,  where $f \colon \Hb^n \setminus U \to \R$ is a continuous function which is smooth on $\Hb^n \setminus \overline{U}$ and $U \subset \Hb^n$ is an open and bounded subset whose complement is connected, such that $\graph(f)$ is 
a balanced asymptotically hyperbolic graph when endowed with the metric induced by $\Hb^{n+1}$, with scalar curvature greater than or equal to $-n(n-1)$, and it is either entire or with minimal boundary. In addition, we require:
\begin{enumerate}[$(1)$]
\item The mean curvature vector of $\graph(f)$ in $\Hb^{n+1}$ points upward
\item For almost every $h \in \mathbb R$, the level set $f^{-1}(h)$ is star-shaped and outer-minimizing in $\Hb^{n}$.
\end{enumerate}
\end{definition}

We denote by $\m(f)$ the total mass of any $\graph(f) \in \G_n$. 
The first named author showed stability of the class $\mathcal{G}_n$ \cite{C}. In particular, after vertically translating $\graph(f)$, 
it was shown that in any ball $B^{\bar b}(\r) \subset \Hb^{n+1}$ of radius $\r$ centered at the origin, 
 $$d_F^{ B^{\bar b}(\r)}(\graph(f), \Hb^n  \times \{ 0 \}) \to 0\quad \text{ as }\,\m(f)  \to 0,$$ 
where $d_{F}$ denotes the Federer--Fleming flat distance. See Theorem \ref{thm-main-C}.

Convergence with respect to the flat distance does not necessarily imply convergence with respect to the intrinsic flat distance. Hence, it is interesting to study the stability of the positive mass theorem with respect to intrinsic flat distance for the class $\G_n$.  
Since by Wenger's compactness theorem, compact oriented Riemannian manifolds 
with a uniform upper diameter bound and a uniform upper volume bound for the manifolds and their boundaries is precompact with respect to intrinsic flat distance, we will consider (in a similar way to  \cite{HLS, HLP}) a subclass of $\mathcal{G}_n$.

\begin{definition}\label{def-Gn2}
For constants $\r_0, \,\g, \, D>0$, we define $\mathcal{G}_n(\rho_0,\g,D)$ to be the space of $n$-dimensional manifolds  $(M,g)$ (possibly with boundary) that admit a smooth Riemannian isometric embedding into $\Hb^{n+1}$, $\Psi:M \to \Hb^{n+1}$, such that $\Psi(M)=\graph(f)$ for some $\graph(f) \in \mathcal{G}_n$ and that in addition the following is satisfied: 
\begin{enumerate}\setcounter{enumi}{2}
\item $U \subset B^b(\rho_0/2)$
\item For $r \geq \tfrac{\r_0}{2}$, a uniform decay condition holds, 
\be\label{eq-Vgrad}
V^2|\grad^b f|_b^2 \leq \g^2
\ee
\item The region $ \Omega(\r_0) = \Psi^{ -1}( \overline{B^b(\r_0)} \times \R)$ has bounded depth,
\bee
\depth(\Omega(\r_0)) = \sup \{ d_M(p,\S(\r_0))\, : \, p \in \Omega(\r_0) \} \leq D,
\eee   
where $d_M$ denotes the length distance in $M$ induced by $g$ and $\S(\r_0) = \pr \Omega(\r_0) \setminus \pr M$. 
\item 
If $f$ is not entire, a stronger condition for the minimal boundary holds, $\nabla^b_{\nu} f \to \infty$ 
 \emph{while $\nabla^b_{X} f $ remains bounded}   as one approaches $\pr U$ (where $\nu$ is the local vector field obtained by extending the outward unit normal of $\pr U$ to a neighborhood of $\pr U$ by parallel transport along the flow lines of the normal exponential map and $X$ is any vector field with $X\perp^b \nu $). We will additionally demand that our minimal boundary is mean convex, i.e., $H\geq 0$, where $H$ denotes the mean curvature of $\pr U$ as a submanifold of $\Hb^{n}$, and star-shaped. 
\end{enumerate}
\end{definition}

Given $(M,g) \in \mathcal{G}_n(\rho_0,\g,D)$ we endow $(\Omega(\r), g|_{\Omega(\r)})$ with the intrinsic length distance, $\di_{\Omega(\r)}$,
and the integral current with weight $1$, $[[\Omega(\r)]]$, as in Example \ref{ex-submanifoldCurrent}.
Our first result is the following. 

\begin{thm}\label{thm-MainO}
Let $M_j\in\G_n(\r_0,\gamma, D)$ be a sequence of asymptotically hyperbolic graphs 
with $\Psi_j:  M_j  \to  \Hb^{n+1}$ a smooth Riemannian isometric embedding as in Definition  \ref{def-Gn2}.
If  $\lim_{j \to \infty}\m(M_j)=0$,   then for any $\r>\r_0$ we have
\bee
\lim_{j \to \infty} d_{\mathcal{F}}   ( (\Omega_{j}(\r), \di_{\Omega_{j}(\r)}, [[\Omega_{j}(\r)]]),  (  \ovr{ B^b(\r)}  , \di_{\ovr{ B^b(\r)}},  [[ \ovr{B^b(\r)} ]]  ))= 0,
\eee
where $d_{\mathcal{F}}$ denotes the intrinsic flat distance, and $\vol(\Omega_{j}(\r)) \to \vol(B^b(\r))$. 
\end{thm}

Conditions (3) and (4) of Definition \ref{def-Gn2} provide a uniform control on the exterior region,  i.e. the complement of $\Omega(\r_0/2)$, of the sequence of manifolds, while (5) prevents the formation of deep ``gravity wells'' (c.f. \cite{LeeSormani1,HLS}). 
Conditions (3)-(5) provide a uniform intrinsic diameter bound of regions $\Omega_j(\r)$
and a uniform bound for $\vol( \bdry \Omega_j(\r))$.  Conditions (1) and (2) together with the uniform diameter bound ensure convergence of $\vol(\Omega_j(\r))$ to $\vol(B^b(\r))$. Thus, Wenger's compactness theorem ensures convergence 
to an integral current space. To get the precise limit space one has to use condition (6) to ensure that the regions 
$\Omega_j(\r)$ embed into suitable Riemannian manifolds diffeormorphic to $\overline{B^b(\r)}$ via a capping procedure as in \cite{HLP},  Theorem \ref{thm-NewOmega}, in order to apply a compactness result by Allen and the third named author \cite{A-P} (see Theorem \ref{convBdry}). We note that to apply Theorem \ref{convBdry} one has to endow $\Omega(\r)$ with the structure $(\Omega(\r), \di_{\Omega(\r)}, [[\Omega(\r)]])$ 
rather than $(\Omega(\r), d_{M}, [[\Omega(\r)]])$. The same is true when applying the compactness result from \cite{A-P} in the proof of \cite[Theorem 3.2]{HLP}. However, by \cite{dNP}  the result in \cite{HLS} is correct for $(\Omega(\r), d_{M}, [[\Omega(\r)]])$ as well.
We also remark that boundedness of $\nabla_X f$ for \emph{tangential directions $X$} as one approaches $\pr U$ as specified in (6) was not originally stated in \cite{HLS}, nor in \cite{HLP}, though in the latter it was implicitly assumed.

We also obtain a pointed version. This is the analogue of \cite[Theorem 1.4]{HLS}.

\begin{thm}\label{thm-MainB}
Let $M_j\in\G_n(\r_0,\gamma, D)$ be a sequence of asymptotically hyperbolic graph manifolds with $\lim_{j \to \infty}\m(M_j)=0$
and $p_j \in \Sigma_j(\r_0)$  be a sequence of points.  Then  for  almost every $R>0$ we have
\bee
\lim_{j \to \infty} d_{\mathcal{F}}   (  (\overline{B^{M_{j}} (p_{j}, R)} , d_{M_j},  [[ \overline{B^{M_{j}} (p_{j},R) }]])  ,   
( \ovr{B^b(R)} , d_{\Hb^n},  [[ \ovr{B^b(R)} ]] )  )  = 0
\eee
and $\vol(B^{M_j}(p_j,R)) \to \vol(B^b(R))$. 
\end{thm}

This manuscript is organized as follows. In Section \ref{sec-back} we provide background material. In Section \ref{sec:prelimResults}
 we prove volume estimates for the regions $\Omega(\r)$, uniform diameter and area bounds and also show their volume converges to the volume of a ball in hyperbolic space provided $\m(f) \to 0$. Additionally, we show Gromov-Hausdoff and intrinsic flat convergence of annular regions, $\Omega(\r') \setminus \Omega(\r)$, and that the inner boundaries, $\bdry M$, converge to the zero integral current space.   The proofs of the main results are given in Section \ref{s-Proofs}.  The proof of Theorem \ref{thm-MainO}, in the entire case, consists in applying the compactness theorem of \cite{A-P}, Theorem \ref{convBdry}, to the regions $\Omega(\r)$, and in the non-entire case, we apply Theorem \ref{thm-NewOmega}, proven 
 in Appendix \ref{sec-appendix}, where we enlarge the $\Omega(\r)$'s and use condition (6) to carefully construct diffeormorphisms from the enlargements to $\overline{B^b(\r)}$, so that Theorem \ref{convBdry} can be applied. Theorem \ref{thm-MainB} follows from Theorem \ref{thm-MainO} and Lemma \ref{theorem:point-convergenceH}, we believe the latter result is interesting in its own. It easily follows from results in \cite{HLP}, see Theorem \ref{theorem:point-convergence}, but in Lemma \ref{theorem:point-convergenceH} 
   we clearly see when Gromov-Hausdorff and intrinsic flat convergence of a sequence of subsets of an intrinsic flat converging sequence
   imply subconvergence of a sequence of points.

\setcounter{tocdepth}{1}
\tableofcontents

\bigskip


 {\bf Acknowledgments.} The authors would like to thank Lan-Hsuan Huang, Dan Lee and Christina Sormani for helpful and interesting discussions during the preparation of this work. AJCP is grateful for the generous support of the Carl Zeiss Foundation and the financial support of the Deutsche Forschungsgemeinschaft through the SPP 2026 “Geometry at Infinity”. MG is grateful for the support of the Deutsche Forschungsgemeinschaft through the SPP 2026 “Geometry at Infinity”. RP acknowledges support from CONACyT Ciencia de Frontera 2019  CF217392 grant. This project started during the Simons Center Mass in General Relativity Workshop, March 26-30, 2018 organized by Christina Sormani, Shing-Tung Yau, Richard Schoen, Mu-Tao Wang  and Piotr Chrusciel to whom we are grateful.

\section{Background}\label{sec-back}

In this section we first collect some results from \cite{D-G-S} about asymptotic hyperbolic graphs and review material from \cite{C} about the class $\mathcal G_n$. In the second part, we define integral current spaces, intrinsic flat distance and state some results that we will apply in subsequent sections.


\subsection{Asymptotically hyperbolic graphs}\label{ssec-AHgraphs}

Here we give some definitions related to asymptotically hyperbolic graphs and state a Riemannian Penrose like inequality, obtained in \cite{D-G-S}.
For a more detailed discussion about asymptotically hyperbolic graphs the reader is referred to  \cite{D-G-S}.

\bigskip
Let $\Hb^n$ denote the hyperbolic space of dimension $n$ and let $b$ be its Riemannian metric, 
which in spherical coordinates $(r,\theta) \in  [0,\infty) \times \bS^{n-1}$ takes the form
\bee
b = dr^2 + \sinh^2(r) \s,
\eee
where $\s$ represents the standard Riemannian metric of $\bS^{n-1}$. 

We will consider graphs of functions over subsets of $\Hb^{n}$ inside the $(n+1)$-dimensional hyperbolic space $\Hb^{n+1}$ with Riemannian metric
$\bar b$, written in coordinates $(r, \theta, s) \in \Hb^n \times \R $ as  
\bee
\bar b = b + V(r)^2 ds^2,
\eee
where $V(r)=\cosh(r)$. Note that the scalar curvature of $\bar b$ is $R(\bar b) = -(n+1)n$.

Given an open set $U \subset \Hb^n$ and a continuous function $f \colon \Hb^n \setminus U \longrightarrow \R$ which is smooth on $\Hb^{n} \setminus \ovr{U}$, we endow
\bee
\graph(f) \defeq \Big\{ (x,   f(x)) \in   \Hb^{n+1}  \, \big | \  x \in \Hb^{n}  \setminus \ovr{U} \Big\}
\eee
with the coordinate chart 
\begin{equation}\label{chart}
\Pi \colon \graph(f)  \to \Hb^{n} \setminus \ovr{U}, \qquad  \Pi(x,f(x)) = x,
\end{equation}
and the Riemannian metric $g$ induced by $\Hb^{n+1}$. To simplify notation, we will often denote geometric quantities associated to $\graph(f)$ using $f$ instead of its metric. For example, we denote its scalar curvature as $R(f)$. This should not cause any confusion as the meaning of the symbols will be clear from the context.

While there is a more general definition of asymptotically hyperbolic manifolds and their mass (see \cite{C-H}, \cite{Wang}) 
we will only consider asymptotically hyperbolic graphs and so, in the interest of simplicity, will only present the graph case here (following \cite{D-G-S}).

\begin{definition}\label{def-AHgraph} 
Let $n\geq 3$ 
and  $U \subset \Hb^n$ be an open (possibly empty)  bounded subset with connected complement. Let  $f \colon \Hb^n \setminus U \longrightarrow \R$ be a continuous function which is smooth on $\Hb^{n} \setminus \ovr{U}$.  We say that $(\graph(f), g)$ is an asymptotically hyperbolic graph (or simply $f$ is asymptotically hyperbolic) with respect to the chart $\Pi$ as in \eqref{chart} if $e=g-b= V^2 df \otimes df$ satisfies 
\begin{enumerate}[(i)]
\item 
\bee
\int_{\Hb^n \setminus B} (|e|_b^2 + |\grad^b e |_b^2) \cosh(r) \, \dvol_b < \infty,
\eee
and
\bee
\dint_{\Hb^n \setminus B} | R(f) + n(n-1)| \cosh(r) \, \dvol_b < \infty,
\eee
where $B$ is a closed ball in $\Hb^n$ that properly contains $U$ and $\dvol_b$ denotes the volume form induced by $b$.
\item  $|e|_b^2=V^2 |\grad^b f|_b^2 \to 0$ at infinity. 
\end{enumerate}
\end{definition}

In general, the total mass of an asymptotically hyperbolic manifold can be defined as the minimization of a functional, called the \emph{mass functional}, that depends on its coordinate chart at infinity, say $\Psi$. If the mass functional is positive over an appropriate subset of a vector space, then $\Psi$ can be chosen so that the mass takes a simpler form; this suitable diffeomorphism $\Psi$ is then referred to as a set of balanced coordinates \cite{D-G-S}.  For an asymptotically hyperbolic graph, if $\Pi$ as in (\ref{chart}) is a set of balanced coordinates we say that $f$ is balanced. In this case the mass takes the form given below \cite{D-G-S}.

\begin{definition} \label{mass-AHgraph}
If $f$ is an asymptotically hyperbolic and balanced function, its mass is given by 
\be \label{AHG-mass}
\m(f) = \dfrac{1}{2(n-1) \w_{n-1}} \lm{r}{\infty} \int_{S_r} (V(\dv^b e - d \tr^b e) + (\tr^b e) dV - e(\grad^b V, \cdot))(\nu_r) \dvol_b,
\ee
where $e \defeq V^2 df \otimes df$, $V(r) =\cosh(r)$, $S_r$ is the coordinate sphere of radius $r$ in $\Hb^{n}$, $\nu_r$ the outward normal vector to $S_r$ and $\w_{n-1}$ denotes the volume of the round sphere $\bS^{n-1}$.
\end{definition}

\begin{definition}
We say that an asymptotically hyperbolic function $f \colon \Hb^n\setminus U \to \R$  is entire
if $U = \emptyset$.  Moreover, we say that $f$ has a minimal boundary if $\pr U \neq \emptyset$, $f$ is constant on each component of $\pr U$ and $|\nabla^bf|_b\to \infty$ as one approaches $\pr U$.
\end{definition}

We now give an example of an asymptotically hyperbolic graph with minimal (and mean convex) boundary which represents an initial data set for the AdS-Schwarzschild spacetime.

\begin{example}[AdS-Schwarzschild manifolds as graphs]
Recall that the spatial AdS-Schwarz\-schild is defined for $m \geq 0$ and $n \geq 3$, as the manifold $(\rho_+,\infty) \times \bS^{n-1}$ with the metric given by
\bee
g_m = \tfrac{1}{1 + \rho^2 - \tfrac{2m}{\rho^{n-2}}} d\rho^2 + \rho^2 \s,
\eee
where $\rho_+$ is the largest root of 
\bee
 \rho^n + \rho^{n-2} - 2m = 0.
\eee
The mass of this manifold is equal to $m$, and clearly,  $g_0= b$.

To see this manifold as a graph over $\Hb^n$ it is convenient to make the change of variables $\rho = \sinh(r)$ on $\Hb^n$, so that the metric on $\Hb^{n+1}$ is given by 
\bee
\bar b = (1+ \rho^2)ds^2 + \tfrac{1}{1+\rho^2} d\rho^2 + \rho^2 \s.
\eee
Then, the AdS-Schwarzschild graph of mass $m$ is given by the function
\bee
f(\rho) = \int_{\rho_0}^{\rho} \tfrac{1}{\sqrt{1+s^2}}\sqrt{\tfrac{1}{1+s^2 -\tfrac{2m}{s^{n-2}}} - \tfrac{1}{1+s^2}}, 
\eee
that is this function is constant in the theta parameter,
and it can be checked by direct computations that this graph has a minimal boundary at $\rho=\rho_+$. 
\end{example}

We now state a Riemannian Penrose like inequality, which was very useful when proving flat convergence of sequences contained in the class $\mathcal G_n$.
We will use it in the proofs of Theorem \ref{thm-IFbounds} and Lemma \ref{lem-inner-gone}, to establish uniform volume bounds for the regions $ \partial\Omega_j(\r)$, and $U_j$ and to show that the inner boundaries of $\Omega_j(\r)$, $\partial M_j$, converge to the zero integral current space, respectively.

\begin{thm}[Riemannian Penrose like inequality {\cite[Theorem 2.1]{D-G-S}}] \label{RPI}
Suppose that $f \colon \Hb^n \setminus U \to \R$ is a balanced asymptotically hyperbolic graph in $\Hb^{n+1}$ with minimal boundary 
and scalar curvature $R(f) \geq -n(n-1)$. Suppose that $\pr U$ is mean convex (i.e., $H \geq 0$, where $H$ denotes the mean curvature of $\pr U$ in $\Hb^n$) and that $U$ contains an inner ball centered at the origin of radius $r_0$. Then,
\bee
 \vol(\pr U) \leq \tfrac{2\w_{n-1}} {V(r_0)} \m(f) 
\eee
where $V(r)=\cosh(r)$ and $\w_{n-1}$ denotes the volume of the round sphere $\bS^{n-1}$.  In particular, $\vol(\pr U)  \leq 2\w_{n-1} \m(f)$.  
\end{thm}

Let us remark at this point that we will from now on, by slight abuse of notation, use $\vol $ to denote volumes of Riemannian (sub-)manifolds regardless of their dimension and we will suppress specifying the metric in the notation unless there would be inequivalent canonical choices.
\bigskip


\subsection{The class $\mathcal{G}_n$ and stability of the PMT with respect to flat distance}\label{ssec-Gclass}

In \cite{C}  the first named author,  inspired by the work of Huang and Lee \cite{H-L}, defined the class $\mathcal G_n$  and proved  the stability of the hyperbolic positive mass theorem with respect to the flat distance (also known as Federer--Fleming distance).    Here we review parts of this work needed in subsequent sections.

\bigskip
We now give more details  about the conditions appearing in Definition \ref{def-Gn}.    Given an asymptotically hyperbolic function, $f \colon \Hb^n \setminus U \to \R$, let $\ovr{\bH}$ denote the mean curvature vector of $\graph(f)$ inside $\Hb^{n+1}$ and let $n_0 \defeq (\vec{0},1) \in T_p \Hb^{n+1}$, for any $p \in \Hb^{n+1}$. We say that $\ovr{\bH}$ points upward if $\bar{b}(n_0,\ovr{\bH})$  is a non-negative function that does not vanish everywhere. 
The mean curvature vector convention is that deformations in the direction given by this vector decrease volume. In particular, the standard sphere has positive mean curvature with respect to the inner pointing unit normal vector field.   We say that $f^{-1}(h)$ is star-shaped if in some radial coordinates $(r, \theta )$ on $\Hb^n\setminus\{0\} $ the set can be written as a smooth graph $(\rho_{f^{-1}(h)}(\theta), \theta ) : \bS^{n-1} \to (0,\infty)\times \bS^{n-1}\cong \Hb^n\setminus \{0\}$. Note that this in particular implies that $f^{-1}(h)$ is a differentiable
	 $(n-1)$-dimensional submanifold of $\Hb^n$. 
 For a bounded and finite perimeter set $E \subset \Hb^{n}$, we say that $\pr^* E$ (where $\pr^*E$ denotes the reduced boundary of $E$) is outer-minimizing if for any bounded set $F$ containing $E$ we have $\H^{n-1}_b(\pr^* E) \leq P(F) $ where $P$ denotes the perimeter of $F$ and $\H^{n-1}_b$ denotes the $(n-1)$ dimensional Hausdorff measure of $\Hb^n$.

\bigskip

The key idea to prove the stability of the hyperbolic positive mass theorem for  $\mathcal G_n$  was to find a suitable ``height'', $h_0(f)$, which divides any $\graph(f)$ in two parts. In the lower part,
\bee
\Big\{  (x, f(x))  \in \Hb^{n+1} \,| \, x \in \Hb^n\setminus U, \,\, f(x) \leq h_0(f)  \Big\},
\eee
one can show that all 
 level sets of $f$ have volume bounded above by some function that depends on $\m(f)$  which goes to zero as $\m(f)$ goes to zero. Meanwhile,  in the upper part, the quantity $\sup(f)-h_0(f)$ is bounded above by a function that also depends on $\m(f)$ and goes to zero as $\m(f)$ does.

In order to define this height, one studies the function 
\bee
\V(h)  \defeq  P( \{ x \in \Hb^n \, : \, \bar f(x) < h  \}),
\eee
where $\bar f$ is the extension of $f$ to $\Hb^n$ which is defined to be constant on $U$. Using condition (1) of the definition of  $\mathcal G_n$ (i.e. that the mean curvature vector of $\graph(f)$ points upward), it is shown that there exists $h_{\max} \in \R$ so that $f  <  h_{\max} $ everywhere and that $\V$ is finite for any $h < h_{\max}$.  Using condition (2) it is shown that  $\V$ is non-decreasing.  Therefore, $\V$ is differentiable almost everywhere and a height can be defined.

\begin{definition}[{\cite[Definition 4.1]{C}}]\label{defn-h0}
Let $\b > 1$ be any fixed constant. Furthermore, let $n \geq 3$ and $f$ be a balanced asymptotically hyperbolic function. We define the height $h_0(f)$ of $f$ as 
\bee \label{h0}
h_0 (f) \defeq \sup\{ h \, : \, \H^{n-1}(f^{-1}(h)) \leq \max \{ 2 \beta \w_{n-1} \m(f)^{\tfrac{n-1}{n-2}},2 \beta \w_{n-1} \m(f) \} \},
\eee
if the above set is non-empty and $h_0(f) = \min(f)$ otherwise. 
\end{definition}
If $\m(f) < 1$ then it follows that  $h_0 (f) = \sup\{ h \, : \, \H^{n-1}(f^{-1}(h)) \leq 2 \beta \w_{n-1} \m(f) \}$.

\bigskip

After an appropriate rescaling of $f$,  using that the mean curvature of the level sets is nonnegative,  
a suitable expression for $\m(f)$ and the Minkowski inequality, one shows that $\V'(h)  \geq F(\V(h))$ for almost every $h \geq h_0(f)$ and some function $F$.
Then the upper bound for $\sup(f)-h_0(f)$  is obtained by comparing $\V$ to the solution of the equation $Y'(h)=  F(Y(h))$ with initial condition equal to $\V(h_0(f))$. 
That is, one finds that $Y \leq \V$ and $Y$ goes to infinity at a finite height, implying that $\V$ also does and hence giving an upper bound to the rescaling of $f$. Rescaling back gives the desired inequality:

\begin{thm}[{\cite[Lemma 4.8]{C}}]\label{thm-fUpperB}
Let $f  \in \mathcal G_n$,  then there exists a constant $C=C(n)$ such that
\be\label{eq-fUpperB}
0 < \sup(f) - h_0(f) < C \m(f)^{\tfrac{1}{n-2}}.
\ee
\end{thm}

The stability of the hyperbolic positive mass theorem with respect to the flat distance reads as follows.

\begin{thm}[{\cite[Theorem 5.1]{C}}] \label{thm-main-C}
Let $n \geq 3$ and $f_j \in \mathcal G_n$ be a sequence of balanced asymptotically hyperbolic functions.
Assume that  $\lim_{j \to \infty}\m(f_j) = 0$.  Then after normalizing $\graph(f_j)$ so that $h_0(f_j) = 0$, we have 
\bee
\lim_{j \to \infty} d_{F}^{B^{\bar b}(\r)}([[\graph(f_j)]], [[\Hb^n \times \{ 0 \}]])=0,
\eee
where for any $\rho >0$, $B^{\bar b}(\r)  \subset \Hb^{n+1}$ is the $\bar b$-ball of radius $\rho$ centered at the origin, 
i.e., in terms of our coordinate system,  $B^{\bar b}(\r) =  \{  (r,\theta,s) \, |\,   \cosh^2(r)s^2 + \sinh^2(r) \leq \rho^2 \}$.
\end{thm}

Theorem \ref{thm-main-C} is proven  by explicitly choosing integral currents  $A_j$ and $B_j$ such that 
$A_j+ \partial B_j= [[\graph(f_j)]]- [[\Hb^n  \times \{ 0 \}]]$ in $B^{\bar b}(\r)$ and 
 $\mass(A_j) + \mass(B_j) \to 0$ as $j  \to 0$, which implies the conclusion. 
All this can be guaranteed by applying Theorem \ref{thm-fUpperB},  the isoperimetric inequality,  the definition of $h_0(f)$ and Theorem \ref{RPI}.


\subsection{Integral currents, intrinsic flat distance and convergence of balls}\label{ssec-currents}

We now give a brief introduction to integral currents in metric spaces, integral current spaces and intrinsic flat distance.
For further details about integral currents in metric spaces we refer the reader to Ambrosio and Kirchheim \cite{AK}, Lang \cite{Lang}, and Lang and Wenger \cite{LaWe}. For the definition of integral current spaces and intrinsic flat distance between them we refer to Sormani and Wenger \cite{SW1,SW2}. For results about point convergence we refer to Sormani \cite{Sormani-AA} and, Huang, Lee and Perales \cite{HLP}.  
\bigskip

\subsubsection{Integral currents}\label{ssec-IntCur}

Given a complete metric space $(Z,d)$,  let  $\lip(Z)$  denote the set of Lipschitz functions on $Z$ and 
$\lip_b(Z)$ the bounded ones.  An $n$-dimensional current $T$ on $Z$ 
is a multilinear map $T: \lip_b(Z) \times [\lip(Z)]^{n}  \to \R$  that satisfies certain properties, see \cite[Definition 3.1]{AK}.  
The $n$-dimensional current $T$ endows $Z$ with a finite Borel measure, $||T||$,
called the mass measure of $T$ and  $\set(T)$ is the set of points in $Z$ where the $n$-dimensional lower density of 
$||T||$ is positive: 
\begin{equation}\label{def-set}
\set(T)= \left\{ z \in Z \, | \, \liminf_{r \downarrow 0} \tfrac{\|T\|(B(z,r))}{ r^n }> 0 \right\}.
\end{equation}
The mass of $T$ is defined as $\mass(T)=||T||(Z)$ and it is known that $\spt(T)=\spt(||T||)=\ovr{\set(T)}$.  For any Lipschitz function $\varphi: Z \to Y$
the push-forward of $T$ is the current ${\varphi}_{\sharp} T: \lip_b(Y) \times [\lip(Y)]^{n}  \to \R$ defined as
\bee
{\varphi}_{\sharp} T (f, \pi_1, \dots, \pi_n) 
= T( f\circ \varphi, \pi_1\circ\varphi,  \dots, \pi_n\circ\varphi).
\eee
The  boundary of $T$,  $\partial T: \lip_b(Z) \times [\lip(Z)]^{n-1}  \to \R$,  is the functional defined as
\bee
\partial T(f, \pi_1,...,\pi_{n-1}) = T(1, f, \pi_1,...,\pi_{n-1}),
\eee
where $1:  Z \to \R$  denotes the constant function equal to $1$.
For any Borel set $A \subset Z$, the restriction of $T$ to $A$, is the current $ T \rstr A: \lip_b(Z) \times [\lip(Z)]^{n}  \to \R$
given by 
\bee
 T   \rstr A ( f, \pi_1,...,\pi_{n}) = T( 1_A \, f, \pi_1,...,\pi_{n}), 
\eee
where $1_A:  Z \to \R$  denotes the indicator function of $A$. In this case $||T   \rstr A|| = ||T ||  \rstr A$ and so $\spt(T \rstr A) \subset \overline A$.

\begin{remark}\label{rmrk-restrictDom}
By  \cite[Proposition 3.3]{Lang}, the current $ T   \rstr A $ can be identified with a current defined in $(\overline{A}, d)$ and that we will denote in the same way,
that is,  $T\rstr A: \lip_b(\overline A) \times [\lip( \overline A)]^{n} \to \R$
(see also  \cite[Equation (3.6)]{AK}).
\end{remark}

The main examples of currents are the zero $n$-dimensional currents, that is, $T(f, \pi_1, \dots, \pi_n )=0$ for all $(f,\pi_1, \dots, \pi_n)$, and the ones given by 
\bee
\varphi_{\sharp} [[\theta]]
(f, \pi_1,..., \pi_n)=  \int_{A} \theta(x) f(\varphi(x))  \det( D_x(\pi_1\circ\varphi,  \dots , \pi_n\circ\varphi)) d\mathcal L^n(x),
\eee
where  $\varphi: A \to  Z$  is a Lipschitz function, $A \subset \R^n$ is a Borel set and, $\theta \in L^1(A, \mathbb R)$.
We will work with  $n$-dimensional integral currents $T$ which are $n$-dimensional currents that can be written as a sum 
of currents of the form $\varphi_{i \sharp} [[\theta_i]]$ as given above, with $\theta_i's$ integer valued, and so that $\bdry T$ is also a current.  The class consisting of these currents will be denoted as
$I_n(Z)$ and to be more precise we will sometimes write $I_n(Z, d)$.

\subsubsection{Integral current spaces and intrinsic flat distance}

An $n$-dimensional integral current space $Q=(X, d_X, T)$ consists of a metric space $(X, d_X)$ and an $n$-dimensional integral current, $T \in I_n(\ovr{X}, d_{\ovr{X}})$,
where $(\overline{X}, d_{\ovr{X}})$ is the metric completion of $(X,d_X)$, and such that $\set(T)=X$.   There is the notion of zero $n$-dimensional integral current space denoted as ${\bf{0} }=(X,d,T)$, here $T=0$ and $\set(T)=\emptyset$. We define $\mass(Q)=\mass(T)$, $\set(Q)=\set(T)$.

The boundary of $Q$, $\partial Q$, is an $(n-1)$-dimensional integral current space and is defined in the following way.
By Remark \ref{rmrk-restrictDom}, $\partial T : \lip_b(\ovr X ) \times [\lip(\ovr X)]^{n-1}  \to \R$ can be identified with a current that we denote in the same way, 
$\partial T : \lip_b(\spt( \bdry T )) \times [\lip(\spt( \bdry T) )]^{n-1}  \to \R$. With this identification
$$\bdry Q= \bdry (X,d_X,T) := (\set(\partial T), d_{\overline X}, \partial T ) \in I_{n-1}(  \spt(\bdry T),d_{\ovr{X}}).$$
 We remark that $\set(\partial T) \subset \ovr{X}$, and that the second entry of $\bdry Q$ is the metric of $\ovr{X}$ restricted to $\set(\partial T)$, which by abuse of notation we write as
 $d_{\overline X}$.

\begin{example}\label{ex-manifold-bdryCurrent}
Given an $n$-dimensional compact oriented Riemannian manifold $(M,g)$, we can define several integral current spaces.

\begin{enumerate}[(i)]
\item $(M,g)$ can be regarded as an $n$-dimensional integral current space,  $(M,d_M,[[M]])$, so that the mass measure of $[[M]]$ equals $\dvol_g$ and $\set([[M]]) = M$.
Indeed, let $d_M$ be the length metric induced by the Riemannian metric $g$.  Choose a locally finite atlas $\{(U_i, \varphi_i)\}_{i \in \mathbb N}$ of $M$ consisting of positively oriented Lipschitz charts and let $\{\psi_i\}$ be a Lipschitz partition of unity of $M$ with respect to this atlas
such that $\spt (\psi_i) \subset U_i  \subset M$. Then for $(f,\pi) \in  \lip_b(M) \times [\lip(M)]^{n}$ we define
\begin{align*}\label{eq-canonicalT}
[[M]] (f,\pi)= &  \sum_{i=1}^\infty {\varphi_i^{-1}}_\sharp [[\psi_i \circ \varphi_i^{-1}]] (f, \pi)\\
= & \sum_{i=1}^\infty  \int_{\varphi_i(U_i)}  ( \psi_i \circ \varphi_i^{-1}) (f \circ \varphi_i^{-1} )  \det( D(\pi \circ \varphi_i^{-1})) d\mathcal L^n. \nonumber
\end{align*}

\item $\partial M$ can be regarded as an $(n-1)$-dimensional integral current space $(\partial M,  d_M,  [[\partial M]] )$. 
Taking $(f,\pi) \in  \lip_b(M) \times [\lip(M)]^{n-1}$,  
\begin{align*}
\partial [[M]] (f,\pi)= &  \sum_{i=1}^\infty {\varphi_i^{-1}}_\sharp [[\psi_i \circ \varphi_i^{-1}]] (1, f, \pi)\\
= & \sum_{i=1}^\infty  \int_{\varphi_i(U_i)}  \psi_i \circ \varphi_i^{-1}(x) \det( D_x((f,\pi) \circ \varphi_i^{-1})) d\mathcal L^n(x) \\
= & \sum_{i=1}^\infty  \int_{\varphi_i(U_i \cap \partial M)}  (\psi_i \circ \varphi_i^{-1}(x))(f \circ  \varphi_i^{-1}(x)) \det( D_x(\pi \circ \varphi_i^{-1})) d\mathcal L^{n-1}(x) 
\end{align*}
so by Remark \ref{rmrk-restrictDom} and abusing notation, we write  $[[\partial M]]= \partial [[M]] \rstr \partial M \in I_{n-1}(\partial M, d_M)$, i.e., the orientation of $M$ passes to $\partial M$ and the atlas of bilipschitz maps are taken with respect to the distance of $M$, and so we write $\partial (M,d_M, [[M]])= (\partial M,  d_M,  [[\partial M]] )$.

\item Given an $n$-dimensional compact submanifold $Y$ of $M$, with possibly non-empty boundary and finite area, we can endow it with the restricted Riemannian metric, $g|_{Y}$,
and the orientation of $M$.  Then as in $(i)$ we get an integral current space $(Y, \di_Y, [[Y]])$ so that $[[Y]] \in I_n(Y, \di_Y)$. We reiterate our notational convention of denoting the intrinsic length distance on $Y$, obtained from considering $(Y,g|_{Y})$ as a Riemannian manifold, by $\di_Y$. If $Y$ is a totally geodesic submanifold of $M$, then $d_M=\di_Y$, but in general $d_M\neq \di_Y$ and it will be important for us to keep track of which distance we are using.

\item If we do not endow $M$ with a Riemannian metric we can still define $[[M]] \rstr Y \in I_n(M, d_M)$, which is given by
\bee
[[M]]\rstr Y (f,\pi)=  \sum_{i=1}^\infty  \int_{\varphi_i(U_i \cap Y)}  ( \psi_i \circ \varphi_i^{-1}) (f \circ \varphi_i^{-1} )  \det( D(\pi \circ \varphi_i^{-1})) d\mathcal L^n. \nonumber
\eee
By Remark \ref{rmrk-restrictDom} and abusing notation, 
 we let $[[Y]]:= [[M]] \rstr Y   \in  I_n(Y, d_M)$
 i.e., the orientation of $M$ passes to $Y$ and the atlas of bilipschitz maps are taken with respect to the distance of M. 
Then  $(Y,d_M,[[Y]])$ is an $n$-dimensional integral current space. 
\end{enumerate}
\end{example}

\begin{example}\label{ex-submanifoldCurrent}
In this article we will work with complete, oriented, non compact Riemannian manifolds $(M^n,g)$. Note that we cannot define $[[M]]$ as in the previous example because Ambrosio--Kirccheim currents require the mass measure to be finite and for $M=\Hb^n, \graph(f)$, for example, this would not be true, though there exist other definitions of currents by Lang and Lang--Wenger \cite{Lang, LaWe} that allow currents to have locally finite measure.
Nonetheless, for $n$-dimensional compact submanifolds $Y  \subset M$ with possibly non-empty boundary we can consider $(Y, \di_Y, [[Y]])$, with  $[[Y]] \in I_n(Y, \di_Y)$, and
$(Y,d_M,[[Y]])$, where $[[Y]]  \in  I_n(Y, d_M)$, as in the previous example. In the latter, we still have $\set([[Y]]) = Y$ and $\partial (Y,d_M,[[Y]])= (\partial Y,  d_M,  [[\partial Y]])$ with $ [[\partial Y]] \in  I_{n-1}(\partial Y, d_M)$.
 	\end{example}

	\begin{example}[{\cite[Lemma 2.34, Lemma 2.35, Remark 2.37]{Sormani-AA}}]\label{lem-ball}
Given an $n$-dimensional integral current space $Q=(X, d_X, T)$. Then for any $p \in X$ and for almost every $r>0$, 
$||T||( d_p^{-1} (r)) + ||\partial T||( d_p^{-1} (r))=0$. For those $r$ it holds 
\begin{itemize}
\item $T \rstr B(p,r)= T \rstr \overline{B(p,r)}$
\item $S(p,r) : = ( \set(T \rstr B(p,r)), d_X, T \rstr B(p,r)) \in I_n(\overline{B(p,r)}, d_{\ovr X} )$
\item $B(p,r) \subset \set(T \rstr B(p,r)) \subset \overline{B(p,r)}$
\item $S^c(p,r):= (\set(T \rstr X\setminus B(p,r)), d_X, T \rstr X\setminus B(p,r) ) \in  I_n( X \setminus B(p,r), d_{\ovr X} ),$
\end{itemize}
where $d_p: X \to \R$ is the function $d_p(x)=d_X(p,x)$, $x \in X$, and we had use Remark \ref{rmrk-restrictDom}.

Furthermore, for $Q=(M, d_M, [[M]])$ as in Example \ref{ex-manifold-bdryCurrent}, 
 $\set(T \rstr B(p,r))=\overline{B(p,r)}$. Hence, for a.e. $r>0$,  $S(p,r) = (\overline{B(p,r)}, d_X, T \rstr \overline{B(p,r)})$.
\end{example}

We say that an integral current space $(X,d,T)$ is precompact if $(X,d)$ is precompact. The definition of intrinsic flat distance is as follows.

\begin{definition}[{\cite[Definition 1.1]{SW2}}]\label{IF-defn} 
Given two $n$-dimensional precompact integral current spaces, $(X_1, d_1, T_1)$ and $(X_2, d_2,T_2)$, the 
intrinsic flat 
distance between them is defined as
\bee
d_{\mathcal{F}}\left( (X_1, d_1, T_1), (X_2, d_2, T_2)\right)=\inf\left\{d_F^Z(\varphi_{1\sharp}T_1, \varphi_{2\sharp}T_2):\,\,\varphi_j: X_j \to Z \right\},
\eee
where the infimum is taken over all  complete metric
spaces $Z$ and all metric isometric embeddings  $\varphi_j$.
The flat distance between two $n$-dimensional integral currents $T_1, T_2  \in  \intcurr_{n} (Z)$, $d_{F}^Z$, is defined as
\bee
d_{F}^Z( T_1, T_2)=\inf  \left\{  \mass(U)+ \mass(V) \, : \,  U \in \intcurr_{n}(Z), \,\,\,  V   \in  \intcurr_{n+1} (Z),\,\, T_2 -T_1=  U + \bdry V       \right\}.
\eee
\end{definition}

The function $d_{\mathcal{F}}$ is a distance up to current preserving isometries {\cite[Theorem 3.27]{SW2}}. 
Hence,
$$d_{\mathcal{F}}((X_1, d_1, T_1), (X_2, d_2, T_2))=0$$ if and only if  there exists  
$\varphi: X_1 \to X_2$ metric isometry such that $\varphi_\sharp T_1=T_2$. In this case, we identify both  integral currents spaces 
\begin{equation}\label{eq-equalCurrSp}
(X_1, d_1, T_1)= (X_2, d_2, T_2),
\end{equation}
and so $d_{\mathcal{F}}$ is a distance in the space of equivalence classes with this relation. 
If $M_i$ are compact oriented Riemannian manifolds,  $(M_1, d_{M_1}, [[M_1]])= (M_2, d_{M_2}, [[M_2]])$ if and only if
there is a Riemannian isometry between $M_1$ and $M_2$ that preserves their orientation.

Wenger proved the following compactness theorem.

\begin{thm}[{\cite[Theorem 1.2]{Wenger-compactness}}]\label{thm-Wenger-compactness}  
Let $D, V, A > 0$ and let $\{(X_j, d_j, T_j)\}_{j=1}^\infty$
be a sequence of $n$-dimensional integral current spaces 
such that
\bee
\diam(X_j) \le D, \quad \mass(T_j)\le V \textrm{ and } \mass(\partial T_j) \le A.
\eee
Then there exists a subsequence $\{(X_{j_k}, d_{j_k}, T_{j_k})\}_{k=1}^{\infty}$  and an $n$-dimensional 
 integral current space $(X_\infty, d_\infty, T_\infty)$
such that 
\bee
\lim_{k\to\infty}d_{\mathcal{F}}((X_{j_k}, d_{j_k}, T_{j_k}),(X_\infty, d_\infty, T_\infty))=0.
\eee
\end{thm}

We will sometimes use the notation  $(X_{j}, d_{j}, T_{j})   \IFto  (X_\infty, d_\infty, T_\infty)$ to denote convergence of the sequence 
$(X_{j}, d_{j}, T_{j})$  to $(X_\infty, d_\infty, T_\infty)$ in the intrinsic flat sense.  Intrinsic flat converging sequences have intrinsic flat converging boundaries 
and the mass functional is lower semicontinuous with respect to this distance.

\begin{thm}[{\cite[Theorem 4.6, Remark 3.3, Remark 3.22]{SW2}}]\label{converge}
Let  $\{(X_j, d_j, T_j)\}_{j=1}^\infty$ be a sequence of integral current spaces that converges in the intrinsic flat sense to
$(X_\infty, d_\infty, T_\infty)$, then
 \begin{enumerate}
 \item $\lim_{j\to\infty} d_{\mathcal{F}}( \bdry (X_j, d_j, T_j) , \bdry (X_\infty, d_\infty, T_\infty))=0$
\item $\mass(T_\infty)\le \liminf_{j\to\infty}\mass(T_j)$
\item If $\lim_{j \to 0}\mass(T_j)=0$, then $(X_\infty, d_\infty, T_\infty)={\bf{0}}$. 
\end{enumerate}
\end{thm}

\subsubsection{Convergence of points and balls}\label{section:point-convergence}

Convergence of points and "balls" under intrinsic flat distance is more subtle than when using Gromov-Hausdorff distance. With intrinsic flat distance we can have 
sequences of points disappearing in the limit or balls converging to the zero integral current space (see some examples of this in  {\cite[Appendix]{SW2}}).  
Here we present several results from Sormani and Huang--Lee--Perales \cite{Sormani-AA, HLP} that will be used in the proof of Theorem \ref{thm-MainB}. 
For historical reasons we start defining convergence of points for GH converging sequences.

\begin{thm}[{\cite[Section 6]{Gromov-metric}}]\label{thm-GromovEmbedding}
A sequence of compact metric spaces $(Y_j, d_j)$, $j \in \mathbb N$,  converges in Gromov-Hausdorff sense to a compact metric space $(Y_\infty, d_\infty)$ if and only if there exist a  
compact metric space $(Z, d)$ and isometric embeddings $\varphi_j:Y_j \rightarrow Z$,  $j\in\mathbb{N}\cup\{\infty\}$,  such that 
 $\varphi_j(Y_j)$ converges with respect to the Hausdorff distance to  $\varphi_\infty(Y_\infty)$.
 \end{thm}
 
When keeping track of the embeddings and the space, we write: 
\bee
Y_j  \xrightarrow[Z, \varphi_j]{\>\mathrm{GH}\>}   Y_\infty. 
\eee

\begin{definition}[Gromov]
Let $(Y_j, d_j)$, $j \in \mathbb N$, be a sequence of compact metric spaces that converges in Gromov-Hausdorff sense to the compact metric space $(Y_\infty, d_\infty)$.
Let $y_j\in Y_j$, $j \in \mathbb N$, and $y \in Y_\infty$. We say that $\{y_j\}$ converges to $y$, $y_j \to y$, if there exist a compact metric space $(Z,d)$ and isometric 
embeddings $\varphi_j:Y_j \rightarrow Z$,  $j\in\mathbb{N}\cup\{\infty\}$, as in the previous theorem, such that $d_Z(\varphi_j(y_j), \varphi_\infty(y))  \to 0$. 
\end{definition}

Since it will be important to keep track of the embeddings and space, in the previous case, we wite:
\bee
(Y_j, y_j)   \xrightarrow[Z, \varphi_j]{\>\mathrm{GH}\>}   (Y_\infty, y). 
\eee

For intrinsic flat convergence there is a similar result.

\begin{thm}[{\cite[Theorem 4.2]{SW2}}]\label{thm-IFiffF}
A sequence of precompact $n$-dimensional integral current spaces $(Y_j, d_j, T_j)$,
 $j \in \mathbb N$,  converges in intrinsic flat sense to the precompact $n$-dimensional integral current space $(Y_\infty, d_\infty, T_\infty)$
 if and only if there exist a  complete and separable metric space $(Z, d)$ and isometric embeddings $\varphi_j:Y_j \rightarrow Z$, $j\in\mathbb{N}\cup\{\infty\}$,  such that 
$\varphi_{j \sharp}(T_j)\to \varphi_{\infty \sharp}(T_\infty)$ in the flat sense in $Z$.
\end{thm}

Similarly, when keeping track of the embeddings and the metric space,  we write  
\bee
(Y_j, d_j, T_j)  \xrightarrow[Z, \varphi_j]{\>\mathcal{F}\>} (Y_\infty, d_\infty, T_\infty).
\eee

\begin{definition}[{\cite[Definition 3.1]{Sormani-AA}}]
Let $(Y_j, d_j, T_j)$ be a sequence of precompact $n$-dimensional integral current spaces, $j \in \mathbb N$, that converges in intrinsic flat sense to the precompact $n$-dimensional integral current space $(Y_\infty, d_\infty, T_\infty)$. Let $y_j\in Y_j$, $j\in\mathbb{N}$, and $y \in  \overline{Y}_\infty$. We say that $\{y_j\}$ converges to $y$,  $y_j \to y$, if there exist a complete and separable metric space $(Z, d)$ and isometric embeddings $\varphi_j:Y_j \rightarrow Z$, $j\in\mathbb{N}\cup\{\infty\}$, as in the previous theorem,  such that $d_Z(\varphi_j(y_j), \varphi_\infty(y)) \to 0$.
\end{definition}

In this case, to avoid any confusion, we will sometimes write: 
\bee
((Y_j, d_j, T_j) , y_j)  \xrightarrow[Z, \varphi_j]{\>\mathcal{F}\>}  ( (Y_\infty, d_\infty, T_\infty), y).
\eee
Note that $y$ is not necessarily contained in $Y_\infty$ and if this is the case we say that the sequence of points disappears in the limit and that 
$y$ disappeared ({\cite[Definition 3.2]{Sormani-AA}}).

When a sequence converges in both Gromov-Hausdorff and intrinsic flat sense to the same limit space the same
embeddings and metric space can be taken. 

\begin{thm}[{\cite[Theorem~3.20]{SW2}}]\label{prop:GHIF}
Let $(Y_j, d_j, T_j)$ be compact $n$-dimensional integral current spaces, $j\in\mathbb{N}\cup\{\infty\}$. 
Then 
\bee
(Y_j, d_j) \GHto (Y_\infty, d_\infty)  \text{ and } (Y_j, d_j, T_j)   \IFto   (Y_\infty, d_\infty, T_\infty)
\eee 
if and only if there exist a complete and separable metric space $(Z, d)$  and isometric embeddings 
 $\varphi_j:  Y_j  \to Z$ such that 
 \bee
 (Y_j, d_j)  \xrightarrow[Z, \varphi_j]{\>\mathrm{GH}\>} (Y_\infty, d_\infty)  \text{  and  }    (Y_j, d_j, T_j)  \xrightarrow[Z, \varphi_j]{\>\mathcal{F}\>} (Y_\infty, d_\infty, T_\infty). 
 \eee
\end{thm}

Under the assumption of the previous theorem, it is easy to see that for any sequence of points $y_j \in Y_j$, $j \in \mathbb N$, there exists $y \in Y_\infty$ such that 
$y_j \to y$. That is, using the Hausdorff convergence of the compact sets $\varphi_j(Y_j)$ to the compact set $\varphi_\infty(Y_\infty)$,  one can find a convergent subsequence of
$\varphi_j(y_{j})$ to a point $z=\varphi_\infty(y)$ for some  $y \in Y_\infty$. 
  \bigskip

In our main theorems, Theorem \ref{thm-MainO} and Theorem \ref{thm-MainB}, the sequences we consider do not necessarily converge in Gromov-Hausdorff sense, but we will be able to prove 
that the sequence of points $p_j$ in Theorem \ref{thm-MainB} does not disappear in the limit by showing convergence in both Gromov-Hausdorff and intrinsic flat sense of annular subregions
to the same limit space. 

\begin{thm}[{\cite[Theorem A1]{HLS}}, c.f. {\cite[Remark 2.22]{Allen-Bur}}]\label{thm-HLS:GH=IF}
Let $(Y, d, T)$ be an $n$-dimensional integral current space 
and $\lambda, \lambda'>0$.   Suppose that
$d_j$ are metrics on $Y$ such that for all $y, y' \in Y$ we have 
\be\label{d_j}
\lambda^{-1}   d(y,y')   \leq  d_j(y,y')    \leq \lambda' d(y,y').
\ee
Then there exist a subsequence, also denoted $d_j$,
and a metric $d_\infty$ satisfying \eqref{d_j} such that $d_j$ converges uniformly to $d_\infty$, 
\bee
\lim_{j\to \infty} d_{GH}\left((Y, d_j), (Y, d_\infty)\right) =0  \text{   and  }  \lim_{j\to \infty} d_{\mathcal{F}}\left((Y, d_j,T_j), (Y, d_\infty,T_\infty)\right) =0, 
\eee
where $T_j = \iota_{j\sharp}  T$ and $\iota_j: (Y, d) \to (Y, d_j)$ are all identity functions. 
\end{thm}

We recall the following useful lemma about intrinsic flat convergence of balls.

\begin{lemma}[{\cite[Lemma 4.1]{Sormani-AA}}]\label{lem-AASorh} 
Let $Q_j=(X_j, d_j, T_j)$ be a sequence of $n$-dimensional integral current spaces, $j\in \mathbb{N}\cup\{\infty\}$, 
such that 
\bee
(Q_j, y_j)   \xrightarrow[Z, \varphi_j]{\>\mathcal{F}\>}   (Q_\infty, y_\infty),
\eee
then there exists a subsequence $y_{j_k} \in X_{j_k}$ such that for almost
every $r>0$ and all $k \in \mathbb N \cup \{ \infty\}$ the triples 
$S(y_{j_k},r)$ and $S^c(y_{j_k},r)$ are integral currents spaces, and using the same isometric embeddings we have,
  \bee
S(y_{j_k},r)  \xrightarrow[Z, \varphi_{j_k}]{\>\mathcal{F}\>}    S(y_{\infty},r) 
 \eee
 \bee
S^c(y_{j_k},r)  \xrightarrow[Z, \varphi_{j_k}]{\>\mathcal{F}\>}    S^c(y_{\infty},r).
 \eee
 \end{lemma}

 \begin{remark}\label{rmrk-volConv}
  If we further assume in the previous lemma that $\lim_{j \to \infty}\mass(Q_j)=\mass(Q_\infty)$, then $\lim_{k \to \infty} \mass(S(y_{j_k},r))  =  \mass(S(y_{\infty},r))$, 
 i.e. $\lim_{k \to \infty} ||T_{j_k}|| (B(y_{j_k},r)) =||T_{\infty}|| (B(y_{\infty},r))$. 
Indeed, by Theorem \ref{converge} we know that
$$\liminf_{k \to \infty} \mass(S(y_{j_k},r))  \geq  \mass(S(y_{\infty},r))$$
 and  
$$\liminf_{k \to \infty} \mass(S^c(y_{j_k},r))  \geq  \mass(S^c(y_{\infty},r)).$$ Now, by Example \ref{lem-ball} we know that $||T_{j_k}|| (d^{-1}_{y_{j_k}}(r))=0$ and thus the mass measures of $S(y_{j_k},r)$ and $S^c(y_{j_k},r)$ are
$||T_{j_k}|| \rstr B(y_{j_k},r)$ and $||T_{j_k}|| \rstr X_{j_k} \setminus B(y_{j_k},r)$, correspondingly. 
Assume that our assertion does not hold, then the first inequality above 
must be a strict inequality, and we get
\begin{align*}
||T_\infty||(X_{\infty}) = &\liminf_{k \to \infty} ||T_{j_k}||(X_{j_k}) =  \liminf_{k \to \infty}  \{ ||T_{j_k}||(B(y_{j_k},r)) + ||T_{j_k}||(X_j \setminus B(y_{j_k},r)) \}\\
> &   ||T_\infty||(B(y_{\infty},r)) + ||T_\infty||(X_\infty \setminus B(y_{\infty},r))= ||T_\infty||(X_{\infty}),
 \end{align*}
 which is a contradiction. 
 \end{remark}

\bigskip
We now introduce notation similar to Example \ref{ex-submanifoldCurrent} and Example \ref{lem-ball}. Given an integral current space $Q=(X, d, T)$ and a subset $Y\subset X$, we define the restriction of $Q$ to $Y$ by 
 \bee
 Q\rstr Y:=  ( \set(T \rstr Y ), d, T\rstr Y)
 \eee
and note that $T\rstr Y$ is a current but might not be an integral current, so in particular the triple is not necessarily an integral current space.

We now see when a non-disappearing sequence of points $y_j \in Y_j$,  with $Y_j \subset X_j$, is a non-disappearing sequence if we  consider each $y_j$ as an element in $X_j$.

\begin{thm}[{\cite[Theorem 2.9]{HLP}}]\label{theorem:point-convergence}
Let $Q_j=(X_j, d_j, T_j)$ be a sequence of  $n$-dimensional integral current spaces, $j\in \mathbb{N}\cup\{\infty\}$, 
such that  
\bee
Q_j   \xrightarrow[Z, \varphi_j]{\>\mathcal{F}\>} Q_\infty,
\eee
and let $Y_j\subset X_j$  such that  $Q_j\rstr Y_j$ are $n$-dimensional integral current spaces that converge
to some integral current space $N_\infty$, 
\bee
(Q_j\rstr Y_j, y_j)   \xrightarrow[W, \psi_j]{\>\mathcal{F}\>}   (N_\infty, y),
\eee
where $y_j\in Y_j$ and $y \in \overline{\set(N_\infty)}$.
If there exists $r >0$ such that the metric ball $\overline{B^{X_j}(y_j, r)}$ is contained in $Y_j$ for all large $j$, then
there exists a  subsequence $y_{j_k}$ and a point $y_\infty\in \overline{X}_\infty$ 
such that 
\bee\label{eq-NewLimitPoint}
 (Q_{j_k}, y_{j_k})   \xrightarrow[Z, \varphi_{j_k}]{\>\mathcal{F}\>}   (Q_\infty, y_\infty).
 \eee
\end{thm}

 \bigskip

We give a lower bound of the distance between a point in a limit space to the set of the boundary.

\begin{thm}[{\cite[Theorem 2.11]{HLP}}]\label{theorem:GH}
Let $Q_j=(Y_j, d_j, T_j)$ be a sequence of $n$-dimensional integral current spaces, $j\in \mathbb{N}\cup\{\infty\}$, so that 
\bee
(Q_j, y_j) \IFZto (Q_\infty, y_\infty). 
\eee
Suppose that $\partial Q_j  \neq {\bf{0}}$, that we can write $\bdry T_j= R^1_j+ R^2_j$ so that 
$R^1_j, R^2_j \in I_{n-1}(Y_j,d_j)$ and that 
$(\set(R^2_j), d_j, R^2_j )  \IFto \bf{0}$,  then 
\bee\label{item:bound}
d_\infty(y_\infty, \set(\bdry T_\infty))\ge \displaystyle \limsup_{j\to\infty} d_j(y_j, \set(R^1_j)). 
\eee
In particular, $B(y_\infty, r) \cap \set(\bdry T_\infty)= \emptyset$ for all $r < \limsup_{j\to\infty} d_j(y_j, \set(R^1_j))$.
\end{thm}

\begin{lemma}\label{lem-decompositionofcurr}
Let  $\{(Y_j, d_j, T_j)\}_{j=1}^\infty$ be a sequence of $n$-dimensional integral current spaces that converges in the intrinsic flat sense to
$(Y_\infty, d_\infty, T_\infty)$. Assume that $T_j= R_j^1 + R_j^2$ so that 
$R^1_j, R^2_j \in I_{n}(\overline Y_j,d_j)$ and $(\set(R^2_j), d_j, R^2_j )  \IFto \bf{0}$, 
then $(\set(R^1_j), d_j, R^1_j )  \IFto  (Y_\infty, d_\infty, T_\infty)$. 
\end{lemma}

\begin{proof}
By  Theorem \ref{thm-IFiffF} there exist a  complete and separable metric space $(Z, d)$ and isometric embeddings $\varphi_j:Y_j \rightarrow Z$, $j\in\mathbb{N}\cup\{\infty\}$,  such that 
$\varphi_{j \sharp}(T_j)\to \varphi_{\infty \sharp}(T_\infty)$ in the flat sense in $Z$. Thus 
there exist $U_j \in I_n(Z)$ and $V_j \in I_{n+1}(Z)$ so that $T_j - T_\infty = U_j + \bdry V_j$ 
and $\mass(U_j)+\mass(V_j) \to 0$. Then $R_j^1-T_\infty= (U_j-R_j^2) + \bdry V_j$, and since $\mass(U_j - R_j^2)\leq \mass(U_j)  + \mass(R_j^2)$, we obtain the desired convergence by applying Theorem \ref{converge}. 
\end{proof}

\bigskip 
To end up this section we state a result that allow us to pass from subconvergence of balls for almost all radii to convergence of all balls for all radii provided one deals with manifolds and has volume convergence. 

\begin{thm}[{\cite[Theorem 2.12]{HLP}}]\label{theorem:subsequences}
Let $(M_j, g_j)$ be  Riemannian manifolds with $p_j\in M_j$, for $j\in\mathbb{N}\cup\{\infty\}$. Assume that for
every subsequence of $\{p_{j_k}\}_{k \in \mathbb N}$ of $\{p_j\}_{j\in \mathbb N}$ there is a subsequence $\{p_{{j_k}_\ell}\}_{\ell \in \mathbb N}$ such that for almost every $r>0$, 
$\vol(\overline{B(p_{{j_k}_\ell},r)}) \to \vol(\overline{B(p_{\infty},r)})$ and $S(p_{{j_k}_\ell},r)     \IFto   S(p_{\infty},r)$.
Then for all almost every $r>0$, 
we have $\vol(\overline{B(p_j,r)}) \to \vol(\overline{B(p_{\infty},r)})$ and $S(p_j,r)     \IFto   S(p_{\infty},r)$.
\end{thm}


\section{Preliminary results}\label{sec:prelimResults}

Given $(M,g) \in \G_n(\rho_0,\gamma,D)$ and $\r >0$, we set $\Omega(\r) = \Psi^{ -1}( \overline{B^b(\rho)} \times \R)$. 
 In Theorem \ref{thm-MainO} we will show intrinsic flat convergence of integral current spaces $( \Omega(\r), \di_{\Omega(\r)}, [[\Omega(\r)]])$, associated to $(\Omega(\r), g|_{\Omega(\r)})$ as in Example \ref{ex-submanifoldCurrent}, to $( \overline{B^b(\r)}, \di_{\overline{B^b(\r)}}, [[\overline{B^b(\r)}]])$, which is the current associated to 
$(\overline{B^b(\r)}, b|_{\overline{B^b(\r)}})$.  By the following remark we will in general denote the previous space as  $(\overline{B^b(\r)}, d_{\Hb^{n}}, [[ \overline{B^b(\r)}]])$ and $b|_{\overline{B^b(\r)}}$ as $b$. 
Clearly, we have that  $\vol_{b}(A)= \vol_{b|_{\overline{B^b(\r)}}}(A)$ for any $A \subset \overline{B^b(\r)}$ and  $\vol_{g}(A)= \vol_{g|_{\Omega(\r)}}(A)$ for any $A \subset \Omega(\r)$ and thus we will not use any subindex for the volume.

\begin{remark}\label{rmrk-totallyGeodBalls}
All open and closed balls in $(\Hb^n,b)$ are totally geodesic given that $b=  dr^2  + \sinh^2(r)\sigma$ with $r\in [0, \infty]$ and $\sinh$ is increasing. Indeed,  for any piece-wise Lipschitz curve $\gamma(t)=(r(t), \theta(t)):[0,1] \to \Hb^n$ joining any two points in $B^b(\r)$, the curve  $(\tilde r(t), \theta(t))$  with  $\tilde r(t)= \min\{r(t), \max\{r(0),r(1)\}\}$  has smaller length than $\gamma$ and is contained in $B^b(\r)$.  The same holds for  $\overline {B^b(\r)}$. 
\end{remark}

In this section, we first calculate some estimates that will be used in the proof of Theorem \ref{thm-MainO}.
We establish uniform intrinsic diameter bounds for regions of the form $\Omega(\r)$ and their volumes, 
and uniform volume bounds of their boundaries. We also show that for 
any sequence $M_j \in \G_n(\rho_0,\gamma,D)$ such that $\m(M_j) \to 0$, 
the volumes of the $\Omega_j(\r)$'s converge to the volume of the ball $\overline{B^b(\rho)}$.

In the last part we show that the annular regions  $\overline{\Omega_j (\r')}  \setminus \Omega_j(\r)$, $\r_0/2 < \r \leq \r'<\infty$,
converge in Gromov-Hausdorff sense to $ \overline{B^b(\r')}  \setminus B^b(\r)$, and that 
the sequence of inner boundaries of $\Omega_j(\r)$, $\bdry M_j$,  converge to the zero integral current space.
These two convergence results in combination with results in Section \ref{section:point-convergence} will be used in the proof of Theorem \ref{thm-MainB}.


\subsection{Volume estimates}\label{sec-volConv}

Here we get uniform upper volume estimates for $\vol( \Omega(\r))$ in terms of $\vol( \overline{B^b(\rho)})$
and an extra term that depends on the hyperbolic mass $\m(M)$ of $M$.  The key ingredients to obtain the estimates are an isoperimetric inequality for the hyperbolic space, the coarea formula and the fact that the manifolds are graphs that satisfy the properties listed in Definition \ref{def-Gn2}.

\bigskip

We recall the following isoperimetric inequality applicable to domains in the hyperbolic space.

\begin{prop}[{\cite[Proposition 3]{Yau}}]\label{prop-IsoIneq}
Let $M$ be a complete simply connected $n$-dimensional Riemannian manifold with sectional curvature bounded from above by $-K$, $K>0$. Then, for any compact domain $D \subset M$, 
\bee
\vol(D) \leq \vol(\bdry D) / \sqrt{K}(n-1).
\eee 
\end{prop} 

We also recall  the following useful fact originally stated in the Euclidean case in the proof of Theorem 3.1 in \cite{HLS}.

\begin{prop}\label{prop-outmin}
Let $S=\partial^*E \subset \Hb^n$, for some $E \subset \Hb^n$, be an outer-minimizing hypersurface, where $\partial^*$ denotes the reduced boundary.  Then $\mathcal{H}^{n-1}(S\cap \overline{B^b(\r)})\le \mathcal{H}^{n-1}(\partial B^b(\r))$. 
\end{prop}

\begin{proof}
 Let $S'=\partial^*(E\cup \overline{B^b(\r)})$. Then $\mathcal{H}^{n-1}(S)\le \mathcal{H}^{n-1}(S')$ by the outer-minimizing property of~$S$.    
Removing $S \cap S'$ from $S$ and $S'$, respectively,  we get 
$$\mathcal{H}^{n-1}(S\smallsetminus S') =  \mathcal{H}^{n-1}(S   \smallsetminus (S\cap S'))   \le       \mathcal{H}^{n-1}(S' \smallsetminus (S\cap S'))    \le \mathcal{H}^{n-1}(S'\smallsetminus S).$$    
Finally, note that $S\smallsetminus S'=S\cap \overline{B^b(\r)}$ and $S'\smallsetminus S\subset \partial B^b(\r)$. 
 \end{proof}

\bigskip

\begin{lemma}\label{lem-volO-}
Let $M \in \G_n(\rho_0, \gamma, D)$ be a manifold so that $\Psi(M)=\graph(f)$ and $\m(f)<1$.  Define 
$$\Omega^{-}(\rho) \defeq \Psi^{-1} \left( \overline{B^b(\rho)} \times (-\infty, h_0(f)) \right),$$ where $h_0(f)$
is given in Definition \ref{defn-h0}. Then, for any $\rho \geq \rho_0$ the following holds:
\bee
\vol(\Omega^{-}(\rho)) \leq 2\beta \omega_{n-1} \m(f)  \left(  c_n +  \cosh(\r) |h_0(f)-\min_{\overline{B^b(\rho)}  \setminus  U} f |  \right).
\eee
Here $c_n$ comes from the isoperimetric inequality, i.e. it equals $1/(n-1)$.
\end{lemma}

\begin{proof} 
By standard computations and the coarea formula, 
\begin{align*}
\vol(\Omega^{-}(\rho)) &= \int_{     \overline{B^b(\r)} \cap     \{ f< h_0(f)\}  } \sqrt{1 + V^2|\grad^b f|^2} \, \dvol_b \\ 
&\leq  \int_{      \overline{B^b(\r)} \cap     \{f< h_0(f)\} }     (1 + V|\grad^b f| ) \, \dvol_b \nonumber \\
  &\leq \H^{n}(f^{-1}(-\infty, h_0(f))) + \cosh(\rho) \int_{-\infty}^{h_0(f)} \H^{n-1}(f^{-1}(s) \cap \overline{B^b(\rho)})\,ds.\nonumber
\end{align*}
Now we calculate the first term on the right hand side.
By Definition \ref{defn-h0}, for all regular values $h \leq h_0(f)$ of $f$,
\bee
\H^{n-1}(f^{-1}(h)) \leq 2\beta \omega_{n-1} \m(f).
\eee  
If necessary, taking a non decreasing sequence of regular values $h_i \in \mathbb R$ of $f$ with $\lim_{i \to \infty} h_i= h_0(f)$ and for which $f^{-1}(h_i)$ is star-shaped, 
and applying the isoperimetric inequality, Proposition \ref{prop-IsoIneq}, we have
\begin{align}\label{eq-bound1O-}
\H^n(  f^{-1}(-\infty,  h_0(f))) &= \lm{i}{\infty} \H^{n}(f^{-1}(-\infty, h_i) ) \leq c_n    2\beta \omega_{n-1} \m(f). 
\end{align}
For the second term, we use again the fact that for  any regular value  $h \leq h_0(f)$ of $f$ we have 
$\H^{n-1}(f^{-1}(h)) \leq 2\beta \omega_{n-1} \m(f)$ and, we note that $\min f \leq h_0(f)$. 
Hence, 
\be\label{eq-bound2O-}
\int_{-\infty}^{h_0(f)} \H^{n-1}(f^{-1}(s) \cap \overline{B^b(\rho)})\,ds \leq     2\beta \omega_{n-1} \m(f) | h_0(f)-\min_{\overline{B^b(\rho)}  \setminus  U} f|.
\ee
Adding \eqref{eq-bound1O-} and \eqref{eq-bound2O-}, 
\bee 
\vol(\Omega^{-}(\rho)) \leq   2\beta \omega_{n-1} \m(f) \left( c_n +  \cosh(\r)  |h_0(f)-\min_{\overline{B^b(\rho)}  \setminus  U} f | \right).
\eee  
\end{proof}

\begin{lemma}\label{lem-volO+}
Let $M \in \G_n(\rho_0, \gamma, D)$ be a manifold so that $\Psi(M)=\graph(f)$ and $\m(f)<1$.  Define 
$$\Omega^+(\r) \defeq \Psi^{ -1} \left( \overline{B^b(\rho)} \times [h_0(f),\infty) \right).$$
Then, for any $\rho \geq \rho_0$ the following holds:
\bee
\vol(\Omega^+(\rho)) \leq \vol( \overline{B^b(\rho)}) +  C\m(f)^{1/(n-2)} \cosh(\rho) \vol(\pr B^b(\rho)).
\eee
Here $C=C(n)$ comes from Theorem \ref{thm-fUpperB}.
\end{lemma}

\begin{proof}
We estimate the volume of $\Omega^+(\rho)$ in the same way as in the proof of Lemma \ref{lem-volO-},
\begin{align*}
\vol(\Omega^+(\rho)) &= \int_{  \overline{B^b(\r)} \cap  \{ f >  h_0(f) \}}  \sqrt{1 + V^2|\grad^b f|^2} \, \dvol_b \\
&\leq  \int_{  \overline{B^b(\r)} \cap       \{f >  h_0(f) \}} (1 + V|\grad^b f| ) \, \dvol_b \\  
&\leq \vol( \overline{B^b(\rho)}) + \cosh(\rho) \int_{h_0(f)}^{\infty} \H^{n-1}(f^{-1}(h) \cap \overline{B^b(\rho)})\,dh, 
\end{align*}
Since $f \in \G_n(\rho_0, \gamma, D)$,  almost every level set of $f$ is outer-minimizing. Thus we can apply 
Proposition \ref{prop-outmin}. Hence, $\H^{n-1}(f^{-1}(h) \cap \overline{ B^b(\rho)}) \leq \vol(\pr B^b(\rho))$ almost everywhere. 
Moreover, by Theorem \ref{thm-fUpperB},  $f  <C\m(f)^{1/(n-2)} + h_0(f)$.
It follows that
\bee
\vol(\Omega^+(\rho)) \leq \vol( \overline{B^b(\rho)}) +  C\m(f)^{1/(n-2)} \cosh(\rho) \vol(\pr B^b(\rho)).
\eee
\end{proof}

\begin{coro}\label{cor-limsupvol}
Let $M \in \G_n(\rho_0, \gamma, D)$ be a manifold so that $\m(M) < 1$.
Then for any $\rho \geq \rho_0$, 
\begin{align*}
\vol(\Omega(\rho)) \leq  &  \vol( \overline{B^b(\rho)}) +  C\m(M)^{1/(n-2)} \cosh(\rho) \vol(\pr B^b(\rho)) + \\
& 2\beta \omega_{n-1} \m(M) \left( c_n +  \cosh(\r)   | h_0(f) - \min_{\overline{B^b(\rho)}  \setminus  U} f | \right).
\end{align*}
\end{coro}

\begin{remark}
Note that there is an easier way to obtain a uniform upper bound estimate for $\vol(\Omega(\r))$. Indeed, 
by the coarea formula and Proposition \ref{prop-outmin}, it follows that 
\bee
\vol(\Omega(\r))\le  \vol( \overline{B^b(\r))} +   |\max_{ \overline{B^b(\r)}\smallsetminus U} f - \min_{\overline{B^b(\r)} \smallsetminus U} f| \cosh(\r)\vol(\partial B^b(\r)).
\eee
Nonetheless, this bound does not immediately imply convergence of the $\vol(\Omega(\r))$ to $\vol(\overline{B^b(\r)})$ provided 
 $\m(M) \to 0$.   Since the estimate in Corollary \ref{cor-limsupvol} involves $\m(M)$, this estimate implies the aforementioned volume convergence as we will see in Theorem \ref{thm-IFbounds}. This is important to be able to
 apply Theorem \ref{convBdry} in the proof of  Theorem \ref{thm-MainO}.
 \end{remark}


\subsection{Diameter bounds, area bounds and volume convergence}\label{sec-IFlim}

Now we prove that for any sequence $\{M_j\} \subset \mathcal  G_n(\r_0,\g,D)$ with mass $\m(M_j)$ converging to zero, the sequence $\Omega_j(\rho)$
has uniform intrinsic diameter bounds, the volumes converge to the volume of the ball $\overline{B^b(\rho)}$ and, the boundaries have uniform volume bounds. 
These estimates will be used in the proof of Theorem \ref{thm-MainO} and Lemma \ref{lem-inner-gone}.

\bigskip

\begin{thm}\label{thm-IFbounds}
Let $\{M_j\} \subset \mathcal G_n(\rho_0,\gamma, D)$ be a sequence of manifolds such that $\lim_{j \to \infty}\m(f_j)=0$ and $\r \geq \r_0$. Then there exist 
$D_0(\r_0, \gamma, D, \r), C_0(\gamma, \r) >0$ so that
\be\label{eq-diamBd}
\diam((\Omega_j(\r),\di_{\Omega_j(\r)})) \le D_0(\r_0, \gamma, D, \r), \quad \vol(\partial \Omega_j(\r)) \leq C_0(\gamma, \r),
\ee
where $\diam(\Omega_j(\r),\di_{\Omega_j(\r)})$ is the diameter of $\Omega_j(\r)$ with respect to $\di_{\Omega_j(\r)}$, 
the intrinsic length distance in $\Omega_j(\rho_0)$ induced by $g_j|_{\Omega_j(\rho_0)}$.
\footnote{Note however that then the estimate clearly also holds for the diameter of $\Omega_j(\r)$ with respect to $d_{M_j}$ because $d_{M_j}\leq \di_{\Omega_j(\r)}$.} Further 
\bee
\vol(\Omega_j(\rho))  \to   \vol( \overline{B^b(\rho)}). 
\eee
\end{thm}
\bigskip

\begin{proof}
Recall that each $M_j \in \mathcal G_n(\r_0,\g,D)$ 
satisfies   \eqref{eq-Vgrad}:  $$V^2(r) |\grad^b f_j|^2 (r, \theta)\le \gamma^2  \qquad \forall r \ge \tfrac{\r_0}{2},\theta\in \mathbb{S}^{n-1}.$$
Any two points in $\Psi_j^{-1} (\overline{B^b(\r)} \setminus B^b(\r_0) \times \R) \subset \Omega_j(\r)$ can be connected by first moving radially inward along the graph of $f_j$, then following an arc on $\Sigma_j(\r_0)$ and then connecting radially outward to the desired endpoint. By the above estimate the length of the first and last segment of this curve must each be bounded by $(\r - \r_0) \sqrt{1+ \gamma^2}$ and the length of the middle segment is bounded by $\pi \sinh(\r_0) \sqrt{1+ \gamma^2}$ since it is an arc on $\Sigma_j(\r_0)$, so the total length of the connecting curve is no longer than
$$2(\r - \r_0) \sqrt{1+ \gamma^2} + \pi \sinh(\r_0) \sqrt{1+ \gamma^2}.$$ 
For points contained in $\Psi_j^{-1} (\overline{B^b(\r_0)} \times \R)$, we note that,  while the definition of depth in Definition \ref{def-Gn2} uses the restricted distance $d_M$, one clearly also has
	\bee
	\depth(\Omega_j(\r_0)) = \sup \{ \di_{\Omega_j(\rho_0)}(p,\S_j(\r_0))\, : \, p \in \Omega_j(\r_0) \}.
	\eee   
Thus by the triangle inequality, 
\bee
\diam(\Omega_j(\r),\di_{\Omega_j(\r)}) \le 2D + 2(\r - \r_0) \sqrt{1+ \gamma^2} + \pi \sinh(\r_0) \sqrt{1+ \gamma^2}=:D_0(\r_0, \gamma, D, \r).
\eee
Hence, we have the required uniform upper diameter bound, denoted by $D_0=D_0(\r_0, \gamma, D, \r)$. 
\bigskip

By Theorem \ref{RPI} (and $V\geq 1$) and  \eqref{eq-Vgrad},
we get   
\begin{align*}
\vol(\partial \Omega_j(\r)) &= \vol(\pr M_j) + \vol(\S_j(\r))\\
&\le  2\w_{n-1} \m(f_j) + \sqrt{1+\gamma^2} \vol(\partial B^b(\r)).
\end{align*}
Since $\lim_{j \to \infty}\m(f_j)=0$ it follows that $\vol(\partial \Omega_j(\r))$ is uniformly bounded.

\bigskip
Since $\lim_{j \to \infty}\m(f_j)=0$ we can assume that $\m(M_j) <1$ and apply Corollary \ref{cor-limsupvol}, 
\begin{align*}
\vol(\Omega_j(\rho)) \leq  &  \vol( \overline{B^b(\rho)}) +  C\m(M_j)^{1/(n-2)} \cosh(\rho) \vol(\pr B^b(\rho)) + \\
& 2\beta \omega_{n-1} \m(M_j) ( c_n +  \cosh(\r)   | h_0(f_j) -  \min_{\overline{B^b(\rho)}  \setminus  U_j} f_j | ).
\end{align*}
Let us find a uniform upper bound for $| h_0(f_j) -  \min_{\overline{B^b(\rho)}  \setminus  U_j} f_j | $.  Note that for any $x,y \in \Hb^n \setminus U_j$
$$ |f_j(x) - f_j(y)| \leq d_{\Hb^{n+1}} ((x,f_j(x)), (y,f_j(y))) \leq d_{M_j} ( \Psi_j^{-1}(x, f_j(x)), \Psi_j^{-1}(y, f_j(y)) ).$$
When restricting the previous to $\overline{B^b(\r)} \setminus U_j$ and using $d_{M_j}\leq \di_{\Omega_j(\r)}$ we get, 
\begin{equation}\label{eq-maxmin}
\max_{ \overline{B^b(\r)}\smallsetminus U_j} f_j - \min_{\overline{B^b(\r)}\smallsetminus U_j} f_j  \leq 
\diam(\Omega_j(\r),\di_{\Omega_j(\r)}) \leq D_0 .
\end{equation}
Thus,
\begin{align}\label{eq-V_0}
\vol(\Omega_j(\rho)) \leq  &  \vol( \overline{B^b(\rho)}) +  C\m(M_j)^{1/(n-2)} \cosh(\rho) \vol(\pr B^b(\rho)) + \\
& 2\beta \omega_{n-1} \m(M_j) ( c_n +  \cosh(\r) D_0) \to \vol( \overline{B^b(\rho)}). \nonumber
\end{align}
Finally, since each $M_j$ is isometric to $\graph(f_j)$, we also have 
\begin{align}\label{eq-V_1}
\vol(\Omega_j(\rho)) \geq  &  \vol( \overline{B^b(\rho)}) -  \vol( U_j),
\end{align}
and $\vol(U_j) \to 0$  by Theorem \ref{RPI} and Proposition \ref{prop-IsoIneq}. So $\vol(\Omega_j(\rho))  \to   \vol( \overline{B^b(\rho)})$.
\end{proof}

\subsection{GH convergence of annular regions and IF convergence of boundaries}\label{ssec-GHannular}

In this section we prove two results needed in the proof of Theorem \ref{thm-MainB}.  
In Theorem \ref{thm-MainB} we have a sequence of points $p_j \in \Sigma_j(\r_0)$ and have to prove intrinsic flat convergence of the balls 
$(\overline{B^{M_j}(p_j,R)}, d_{M_j}, [[\overline{B^{M_j}(p_j,R)}]])$ to a ball in hyperbolic space. 
By Theorem \ref{thm-MainO} we will have $(\Omega_j(\widetilde R), \di_{\Omega_j(\widetilde R)},[[\Omega_j(\widetilde R)]])$ converging in intrinsic flat sense to  $(\overline{B^b(\widetilde R)},  d_{\Hb^n},  [[ \overline{B^b(\widetilde R)}]])$. 
To ensure that up to a subsequence $p_j  \to p_\infty$  for some $p_\infty \in  \overline{B^b(\widetilde R)}$, the main ingredient will be
Corollary \ref{lem-Phi} where one obtains
Gromov-Hausdorff and intrinsic flat convergence of the annular regions $A(\r,\r'):=\overline {\Omega(\rho')} \setminus \Omega(\rho)$ 
to $ \overline{B^b(\r')}  \setminus B^b(\r)$. 
Then to ensure that $\overline{B^b(p_\infty, R)}  \subset \overline{B^b(\widetilde R)}$ we will have at our disposal Lemma \ref{lem-inner-gone}
where we show that the sequence of inner boundaries 
$(\bdry M_j,  \di_{\Omega_{j}(\r)},  [[ \bdry M_j]])$ converge in intrinsic flat sense to the zero $(n-1)$-dimensional integral current space.

\bigskip

\begin{lemma}
Assume $M \in \G_n(\rho_0, \gamma, D)$ such that $\Psi(M)=\graph(f)$, for $f \in \G_n$. Write $\Psi=(\Psi^1, \Psi^2):  M \to  \Hb^n \times \R$.
Then for any $\r  > \r_0/2$ there exists a constant $\Gamma=\Gamma(\r_0,\r,\g)>0$, such that 
\be\label{phi-j-r}
\Psi^1: (M \setminus \Omega(\r),d_{M}) \to ( \Hb^{n} \setminus B^b(\r),  d_{\Hb^{n}})
\ee  
is bilipschitz with $Lip(\Psi^1)=1$ and $Lip((\Psi^1)^{-1})=\Gamma$. 
Furthermore, for any $R \geq \r' > \r > \r_0/2$ 
\begin{equation*}
	\Psi^1: (A(\r, \r'), \di_{\Omega(R)}) \to  (\overline{B^b(\r')}  \setminus  B^b(\r) ,  d_{\Hb^{n}} ),
\end{equation*}
is bilipschitz with $\lip(\Psi^1)=1$ and $\lip((\Psi^1)^{-1})=\Gamma(\r_0,\r,\g)$. 
\end{lemma}

\begin{proof}
	Recall that $\Psi:M   \to \Hb^{n+1} \setminus U$ is a smooth Riemannian isometric embedding and that 
	$\Psi(M)= \graph(f)$. Hence, $\Psi^1: (M, d_M) \to (\Hb^{n} \setminus U, d_{\Hb^n})$ is bijective and 1-Lipschitz. To simplify notation, denote the inverse of  $\Psi^1|_{M \setminus \Omega(\r)}:  M \setminus \Omega(\r) \to \Hb^{n} \setminus B^b(\r)$  by $\Phi$.
	We have to show that $\Phi$ is $\Gamma$-Lipschitz.

	Let  $x,x' \in  \Hb^n \setminus  B^b(\r)$ be two points. If the unit-speed $b$-geodesic $c(t):[0, d_{\Hb^{n}} (x,x')] \to \Hb^n$ that connects them lies in $\Hb^n \setminus  B^b(\r)$, then by the Cauchy-Schwartz inequality and \eqref{eq-Vgrad},
	\begin{align}\label{eq-estPhiM}
		d_{M}(\Phi(x), \Phi(x'))  \leq & \int_0^{ d_{\Hb^{n}} (x,x') } \sqrt{ b(c'(t), c'(t)) +  V^2  b^2(c'(t),  \grad f)}  \\
		\leq & d_{\Hb^n}(x,x') \sqrt{1+\g^2} 
		 \nonumber
	\end{align}

	If $c$ passes through $B^b(\r)$ then there exist intervals $[0,t_1],[t_1,t_2],[t_2, d_{\Hb^{n}} (x,x')]$, see Remark \ref{rmrk-totallyGeodBalls},  such that 
	$c([0,t_1]), c( [t_2, d_{\Hb^{n}} (x,x')]  )  \subset  \Hb^n  \setminus B^b(\r)$ and $c([t_1,t_2]) \subset \overline{B^b(\r)}$.  Note that estimate (\ref{eq-estPhiM}) holds for $d_M(\Phi(x),  \Phi(c(t_1)))$ and  $d_M( \Phi(c(t_2)),\Phi(x'))$.

	For the remaining part, note that $d_{\Hb^n}(p, \pr B^b(\frac{\r_0}{2}))= \rho-\frac{\rho_0}{2}$ (as the radial curve provides a minimizing geodesic) for any $p\in \pr B^b(\r) $, so if $d_{\Hb^n}(c(t_1),c(t_2)) \leq \rho-\frac{\rho_0}{2}=:\l_*$, then
	the unit-speed $b$-geodesic from $c(t_1)$ to $c(t_2)$  lies in $\overline{B^b(\rho)}\setminus B^b(\rho_0/2)$.
	Hence, (\ref{eq-estPhiM}) holds for $d_M(\Phi(c(t_1)), \Phi(c(t_2))$.  
If $d_{\Hb^n}(c(t_1), c(t_2)) \geq \l_*$, observe that the pullback metric of $b$ under the inclusion map, $\pr B^b(\r)  \to \Hb^{n}$,  equals the standard round metric in a sphere of radius $\sinh(\r)$. Thus, the diameter of $\pr B^b(\rho) \subset \Hb^n$ with respect to its induced Riemannian metric equals $\pi \sinh(\rho)$ and so $ c(t_1)$ and $ c(t_2)$ can be connected by a curve $\sigma$ in $\pr B^b(\rho)$ of length bounded above by $\pi \sinh(\rho)$. We have
	\begin{align*}
d_M(\Phi(c(t_1)),\Phi(c(t_2) ))\leq & \sqrt{1+\gamma^2} L_{\Hb^n}(\sigma)\leq \frac{\sqrt{1+\gamma^2} \pi \sinh(\rho)}{\l_*} \l_* \\
\leq & \frac{\sqrt{1+\gamma^2} \pi \sinh(\rho)}{\l_*} d_{\Hb^n}(c(t_1), c(t_2)).
\end{align*}

	From the triangle inequality and using that $c$ is a minimizing $b$-geodesic from $x$ to $x'$, we conclude that 
	\begin{equation*}
		d_M(\Phi(x),\Phi(x')) \leq   \max(1, \tfrac{\pi \sinh(\rho)}{\r-\frac{\r_0}{2}})   \sqrt{1+\g^2} d_{\Hb^n}(x,x').
	\end{equation*}
	Then define $\Gamma(\r_0,\r,\g): = \max(1, \tfrac{\pi \sinh(\rho)}{\r-\frac{\r_0}{2}})   \sqrt{1+\g^2}$.
	
	Similarly $\Psi^1 : (\Omega(R), \di_{\Omega(R)}) \to  (\overline{B^b(R)}  \setminus U ,  d_{\Hb^n})$ is $1$-Lipschitz and bijective. 
	Proceeding as in the previous case and taking into account that $\overline{B^b(R)} \subset \Hb^n$ is a totally geodesic submanifold  (see Remark \ref{rmrk-totallyGeodBalls}), the inverse of $\Psi^1$  restricted to the annular region $A(\r, \r')$ 
 is $\Gamma(\r_0/2,\r,\g)$-Lipschitz. 
\end{proof}

By applying  Theorem \ref{thm-HLS:GH=IF} to the sequence $A_j(\r,\r')$ we get the following.

\begin{coro}\label{lem-Phi}
Let $M_j \in \G_n(\rho_0, \gamma, D)$ be a sequence. Then for any $R \geq \r' \geq \r> \r_0/2$ and with the same notation as in the previous lemma, we have
\bee
(A_j(\r, \r'), \di_{\Omega_j(R)}) \GHto (A^b(\r,\r'), d_\infty)
\eee
and
\bee
(A_j(\r, \r'), \di_{\Omega_j(R)},  [[  A_j(\r, \r') ]] ) \IFto (A^b(\r,\r'), d_\infty , T_\infty),
\eee
where $A^b(\r,\r') =\overline{B^b(\r')}  \setminus  B^b(\r)$,  $d_\infty:  A^b(\r,\r') \times A^b(\r,\r') \to \R$ is a distance function that 
satisfies $d_{\Hb^n}(y,y') \leq d_\infty(y,y') \leq \Gamma(\r_0, \r, \gamma)  d_{\Hb^n}(y,y')$ for any $y, y' \in A^b(\r,\r')$, 
$\iota_\infty: (A^b(\r,\r'), d_{\Hb^n})  \to (  A^b(\r,\r'), d_\infty)$  is the identity function and $T_\infty= \iota_{\infty \sharp}  [[A^b(\r, \r')]]$.
\end{coro}

\bigskip

Now we show that the inner boundaries converge to the zero integral current space and that the
outer boundaries converge to the boundary of the limit space.

\begin{lemma}\label{lem-inner-gone}
Let $\{M_j\} \subset \mathcal G_n(\rho_0,\gamma, D)$ be a sequence such that $\lim_{j \to \infty}\m(f_j)=0$ and $\r \geq \r_0$.  
Then there exist a subsequence $\{(\Omega_{j_k}(\r), \di_{\Omega_{j_k}(\r)}, [[\Omega_{j_k}(\r)]] ) \}$ and an integral current space $(\Omega_\infty^\r, d_\infty^\r, T_\infty^\r)$ such that
\bee
(\Omega_{j_k}(\rho), \di_{\Omega_{j_k}(\r)},[[\Omega_{j_k}(\rho)]])   \IFto   (\Omega_\infty^\r, d_\infty^\r , T_\infty^\r).
\eee
With no loss of generality assume that (this is true up to a sign) for all $j \in \mathbb N$,
$$\partial [[\Omega_{j}(\r)]]= [[ \Sigma_j(\r)]] - [[ \bdry M_j]] \in I_{n-1}( \Omega_j(\r), \di_{\Omega_{j}(\r)})$$
with  $[[ \Sigma_j(\r)]], [[ \bdry M_j]] \in I_{n-1}( \Omega_j(\r), \di_{\Omega_{j}(\r)})$ and where $\Sigma_j(\rho)=\partial\Omega_j(\r)\smallsetminus \partial M_j$.  Then we have
\bee
(\partial M_{j_k},   \di_{\Omega_{j_k}(\r)},  [[\partial M_{j_k}]] ) \IFto  {\bf{0}}
\eee
and
\bee
(\Sigma_{j_k}(\rho), \di_{\Omega_{j_k}(\r)},  [[\Sigma_{j_k}(\rho)]] ) \IFto  \partial (\Omega_\infty^\r, d_\infty^\r , T_\infty^\r).
\eee
\end{lemma}

\begin{proof}
By Wenger's compactness theorem, Theorem \ref{thm-Wenger-compactness},  and  Theorem \ref{thm-IFbounds}
 there exist a subsequence $\{ (\Omega_{j_k}(\r), \di_{\Omega_{j_k}(\r)}, [[\Omega_{j_k}(\r)]]) \}$ and an integral current space $(\Omega_\infty^\r, d_\infty^\r, T_\infty^\r)$ such that
$(\Omega_{j_k}(\rho), \di_{\Omega_{j_k}(\r)}, [[\Omega_{j_k}(\rho)]]) $ converges in intrinsic flat sense to $(\Omega_\infty^\r, d_\infty^\r , T_\infty^\r)$.  Then 
by Theorem \ref{converge} we get
\bee
\partial (\Omega_{j_k}(\rho), \di_{\Omega_{j_k}(\r)} ,[[\Omega_{j_k}(\rho)]])   \IFto  \partial (\Omega_\infty^\r, d_\infty^\r , T_\infty^\r).
\eee
Since $\mass ([[  \partial M_j ]])= \vol(\partial M_j)$ and by the Riemannian Penrose like inequality for asymptotically hyperbolic graphs,
Theorem \ref{RPI},  we have $\vol(\partial M_j) \le  2\w_{n-1}\m(f_j) \to 0$.  It follows by Theorem \ref{converge} that $(\partial M_{j_k},   \di_{\Omega_{j_k}(\r)},  [[\partial M_{j_k}]] ) \IFto  {\bf{0}}$. Then by Lemma \ref{lem-decompositionofcurr} we conclude the proof of the lemma.
\end{proof}


\section{Proofs of the main theorems}\label{s-Proofs}

Here we prove Theorem \ref{thm-MainO} and Theorem \ref{thm-MainB}.
A key tool to prove the former is Theorem \ref{convBdry}.  Unlike  Wenger's compactness theorem, Theorem \ref{thm-Wenger-compactness}, Theorem \ref{convBdry}
tell us which space is the limit.  Theorem  \ref{convBdry} was applied in {\cite{A-P, HLP}} 
to fill in a gap in the proof of the stability of the positive mass theorem for asymptotically flat graphical manifolds
under intrinsic flat distance \cite{HLS, HLS-Corr}.

\begin{thm}[{\cite[Theorem 4.2, c.f. Proof of Theorem 1.1 in \cite{A-P}]{A-P}}]\label{convBdry}
Let $D, A >0$ and $\lambda_j \in (0,1]$ be a sequence such that $\liminf_{j\to \infty}{\lambda_j}=1$. 
Let $(\Omega_j,g_j)$ be a sequence of $n$-dimensional compact oriented Riemannian manifolds with non-empty boundary, $j \in \mathbb N \cup \{\infty\}$,  with $g_j$ continuous 
for $j \in \mathbb N$ and  $g_\infty$ smooth, so that any  $g_\infty$-minimizing geodesic between two points in the interior of $\Omega_\infty$
lies completely in the interior.   Assume that for all $j \in \mathbb N$,
\bee
\diam(\Omega_j) \le D, \quad \vol(\partial \Omega_j) \leq A,
\eee
 that
\bee
\vol(\Omega_j) \to \vol(\Omega_\infty)  \quad \text{as  } j \to \infty,
\eee
and that there exist $C^1$  diffeomorphisms $\Phi_j: \Omega_\infty  \to \Omega_j$  so that  for all $v \in T\Omega_\infty$
\bee 
\lambda_j  g_\infty(v,v) \le g_j(d\Phi_jv, d\Phi_jv).
\eee 
Then $(\Omega_j,d_{\Omega_j},[[\Omega_j]])$ converges in intrinsic flat sense to $(\Omega_\infty,d_{\Omega_\infty}, [[\Omega_\infty]])$.
\end{thm}

Let us recall our first main theorem, Theorem \ref{thm-MainO}.

\begin{thm}
Let $M_j\in\G_n(\r_0,\gamma, D)$ be a sequence of asymptotically hyperbolic graphs 
with $\Psi_j:  M_j  \to  \Hb^{n+1}$ a smooth Riemannian isometric embedding as in Definition  \ref{def-Gn2}.
If  $\lim_{j \to \infty}\m(M_j)=0$,   then for any $\r>\r_0$ we have
\bee
\lim_{j \to \infty} d_{\mathcal{F}}   ( (\Omega_{j}(\r), \di_{\Omega_j(\r)}, [[\Omega_{j}(\r)]]),  (  \ovr{B^b(\r)}  , d_{\Hb^n},  [[  \ovr{B^b(\r)} ]]  ))= 0
\eee
and $\vol(\Omega_j(\r)) \to \vol(B^b(\r))$.
\end{thm}

In the non-entire case we prove Theorem \ref{thm-MainO} by enlarging each $\Omega_j(\r)$, as described in Appendix \ref{sec-appendix},
to get manifolds diffeormorphic to $\ovr{B^b(\r)}$ 
and  then apply Theorem \ref{convBdry} to the new sequence.  The uniform diameter, volume and area 
bounds  needed follow by the corresponding uniform diameter and area bounds shown in  Theorem \ref{thm-IFbounds}
and the way the enlargements and diffeomorphisms are chosen.

\begin{proof}[Proof of Theorem \ref{thm-MainO}]
We first assume that all the manifolds $M_j$ have non-empty boundary. 
Let $\lambda_j \in (0,1]$ be a sequence such that $\liminf_{j\to \infty}{\lambda_j}=1$ and fix some  $L> D+\tfrac{1}{2}\sinh(\rho_0)\,\pi \sqrt{1+\gamma^2}$. 
We replace each $(\Omega_j(\r), g_j)$ by a manifold  $(\widetilde\Omega_j, \tilde g_j)$ diffeomorphic to  $\overline{B^b(\r)}$ by applying Theorem \ref{thm-NewOmega} for $\lambda=\lambda_j$. To obtain the conclusion of the first part of the theorem, by the triangle inequality, it is enough to prove that 
 \bee
  \lim_{j \to \infty} d_{\mathcal{F}}((\Omega_{j}(\r), \di_{\Omega_j(\r)}, [[\Omega_{j}(\r)]]),  (\widetilde\Omega_j, d_{\widetilde\Omega_j}, [[\widetilde\Omega_j]]) ) = 0
 \eee
 and 
  \bee
  \lim_{j \to \infty}d_{\mathcal{F}}(  (\widetilde\Omega_j, d_{\widetilde\Omega_j}, [[\widetilde\Omega_j]]) , (  \ovr{B^b(\r)}  , d_{\Hb^n},  [[  \ovr{B^b(\r)} ]]  )) = 0.
  \eee

We get the first limit as follows.  By the definition of intrinsic flat distance and given that by Theorem \ref{thm-NewOmega} $(\Omega_j(\r),\di_{\Omega_j(\r)}) $ embeds in a distance preserving way into $(\widetilde\Omega_j,d_{\widetilde\Omega_j})$, we get
 \bee 
 d_{\mathcal{F}}((\Omega_{j}(\r), \di_{\Omega_j(\r)}, [[\Omega_{j}(\r)]]),  (\widetilde\Omega_j, d_{\widetilde\Omega_j}, [[\widetilde\Omega_j]]) ) \leq  d_{F}^{ \widetilde\Omega_j}([[\Omega_{j}(\r)]], [[\widetilde\Omega_j]]) \leq \vol(\widetilde\Omega_j \setminus \Omega_j(\r)).
 \eee
By  Theorem \ref{thm-NewOmega}, Theorem \ref{RPI} and our hypothesis on the masses we get $ \vol(\widetilde\Omega_j \setminus \Omega_j(\r)) \leq \widetilde V(L, \r_0, \partial U_j) \to 0$. Thus, we get the first limit.

\bigskip 
The second limit is proven by ensuring that we can apply Theorem~\ref{convBdry} to $\{(\widetilde\Omega_j,  d_{\widetilde\Omega_j},  [[ \widetilde\Omega_j]])\}_{j \in \mathbb N}$, 
where the limit space has to be the hyperbolic ball $(  \overline{B^b(\r)} , d_{\Hb^n}, [[  \overline{B^b(\r)}  ]])$.  In Theorem \ref{thm-IFbounds} it was shown that 
for all $j \in \mathbb N$
\begin{align*}
\diam((\Omega_j(\r),\di_{\Omega_j(\r)})) \leq  D_0(\r_0, \gamma, D, \r), \quad  \vol(\partial \Omega_j(\r))  \leq C_0(\gamma, \r),
\end{align*}
and $\lim_{j \to \infty} \vol(\Omega_j(\r)) = \vol(  \overline{B^b(\r)})$.

Recall that $(\Omega_j(\r),\di_{\Omega_j(\r)})$ distance preserving embeds into $(\widetilde \Omega_j, d_{\widetilde\Omega_j})$. Thus, a uniform diameter bound for the $\widetilde \Omega_j$'s 
is 
$$\diam( (\Omega_j(\r),\di_{\Omega_j(\r)}) )  +  \diam ((\widetilde \Omega_j \setminus \Omega_j(\r),d_{\widetilde\Omega_j}) ) \leq D_0(\r_0, \gamma, D, \r)  +  \widetilde D(L, \r_0),$$
where $\widetilde D(L, \r_0)$ is the upper diameter bound provided in Theorem \ref{thm-NewOmega}. 
Now, $\partial \widetilde \Omega_j =   \partial  \Omega_j(\r)  \setminus \bdry M_j$ so $\vol(\partial \widetilde \Omega_j) \leq C_0(\gamma, \r)$.
Since we saw above that $\vol( \widetilde \Omega_j  \setminus \Omega_j(\r)) \leq \widetilde V(L, \r_0, \partial U_j) \to 0$, we get
\begin{align*}
\lim_{j\to \infty}\vol( \widetilde \Omega_j) & = 
\lim_{j \to \infty} \big( \vol(\Omega_j(\r))  + \vol( \widetilde \Omega_j  \setminus \Omega_j(\r)) \big) =  \vol(\overline{B^b(\r)}). 
\end{align*}
Since $\liminf_{j\to \infty}{\lambda_j}=1$ and $\widetilde g_j  \geq \lambda_j b$ by Theorem \ref{thm-NewOmega}, and by Remark
\ref{rmrk-totallyGeodBalls} the ball $\overline{B^b(\r)}$ is totally convex, we can apply Theorem \ref{convBdry} to conclude that the second limit holds. Putting both limits together we get the conclusion,
\bee
\lim_{j \to 0} d_{\mathcal{F}}   ( ( \Omega_{j}(\r), \di_{\Omega_j(\r)}, [[\Omega_{j}(\r)]]),  (  \overline{B^b(\r_0)}, d_{\Hb^n},  [[ \overline{B^b(\r_0)} ]]) )  = 0.
\eee

If the $M_j$ are the graphs over entire functions, then there is no need to enlarge the manifolds. The result follows  immediately from Theorem \ref{thm-IFbounds}  and Theorem~\ref{convBdry}.   If the sequence $M_j$ contains subsequences of both entire and non-entire manifolds the result follows from the previous cases. 
\end{proof}

Next we will rewrite Theorem \ref{thm-MainB} and give its proof. However, before we do this, let us first establish the following lemmas.

\begin{lemma}\label{theorem:point-convergenceH}
Let $Q_j=(X_j, d_j, T_j)$ be a sequence of  $n$-dimensional integral current spaces, $j\in \mathbb{N}\cup\{\infty\}$, 
so that $Q_j \IFto Q_\infty$.
 Suppose that  $y_j \in X_j$ and $Y_j\subset X_j$ are compact sets, $j \in \mathbb N$,  such that  
\begin{itemize}
\item $(Y_j, d_j) \GHto (Y,d)$ for some (non-empty) compact metric space $(Y,d)$
\item  $(Y_j, d_j, T_j \rstr Y_j)$ are $n$-dimensional integral current spaces that converge
to the integral current space $N=(Y,d,S)$, for some $S \in I_n(Y)$
\item there exists $r >0$ such that $\overline{B^{X_j}(y_j, r)}$ is contained  in $Y_j$ for all large $j$.
 \end{itemize}
Then for any complete metric space $Z$ and isometric embeddings $\varphi_j: X_j \to Z$ that satisfy 
\bee
Q_j   \xrightarrow[Z, \varphi_j]{\>\mathcal{F}\>} Q_\infty,
\eee
we can ensure that there exist a  subsequence $y_{j_k}$ and $y_\infty\in \overline{X}_\infty$ such that
\bee\label{eq-NewLimitPointH}
(Q_{j_k}, y_{j_k})   \xrightarrow[Z, \varphi_{j_k}]{\>\mathcal{F}\>}   (Q_\infty, y_\infty) .
\eee
\end{lemma}

\begin{proof}
By Theorem~\ref{prop:GHIF} we have 
        \begin{align*}   
	(Y_j, d_j)  \xrightarrow[W, \psi_j]{\>\mathrm{GH}\>} (Y,d) \quad  \text{and}\quad
	 (Y_j, d_j, T_j \rstr Y_j) \xrightarrow[W, \psi_j]{\>\mathcal{F}\>}  (Y,d,S),
	\end{align*}
	where $W$ is a complete and separable metric space and  $\psi_j$ are isometric embeddings.  
	Since  $\psi_\infty(Y)$ is a compact subset of a complete space, $W$, and 
	\bee
	\psi_j(Y_j)  \xrightarrow[W]{\>\mathrm{H}\>}  \psi_\infty(Y),
	\eee
	there exists a subsequence of $\psi_j(y_j)$  that converges to $\psi_\infty(y)$ for some $y \in Y$. 
	Therefore
	\bee
	 ((Y_j, d_j, T_j \rstr Y_j), y_j) \xrightarrow[W, \psi_j]{\>\mathcal{F}\>} ( (Y,d,S), y).
	\eee
	Then by Theorem \ref{theorem:point-convergence} there exists $y_\infty \in \overline X_\infty$ and a subsequence such that 
	\bee
(Q_{j_k}, y_{j_k})   \xrightarrow[Z, \varphi_{j_k}]{\>\mathcal{F}\>}   (Q_\infty, y_\infty) .
\eee
\end{proof}

The next lemma compares $d_{M}$ and $\di_{\Omega(\r)}$ on relevant regions of $M$ for large $\r$. This is important as it will allow us to use Theorem \ref{thm-MainO} in the proof of Theorem \ref{thm-MainB}.

\begin{lemma}\label{lem-distequiv}
	Let $(M,g)\in\G_n(\r_0,\gamma, D)$ and fix $\bar{R}>1+(4+\pi \sinh(\r_0+1))\sqrt{1+\g^2}$. Then
	\begin{enumerate}
		\item $\di_{\Omega(\r_0+\bar{R})}=d_{M}$ on $\overline{B^{M} (p, R)}$ for all $p\in \Sigma(\r_0)$ and $R<\frac{\bar{R}}{3}$,
		\item $\di_{\Omega(\r_0+\bar{R})}=d_{M}$ on $A(\r_0-1,\r_0+1)$.
	\end{enumerate}
\end{lemma}

\begin{proof}
To show $\di_{\Omega(\r_0+\bar{R})}= d_{M}$ on the desired subsets we need to show that for two points $x,y\in \overline{B^{M} (p, R)}$ (resp.\ $x,y\in A(\r_0-1,\r_0+1)$) any curve of length $d_M(x,y)$ from $x$ to $y$ in $M$ will remain in $\Omega(\r_0+\bar{R})$. This is satisfied if any curve of length bounded by $\diam((\overline{B^{M} (p, R)},d_M))\leq 2R$ (resp.\ $\diam((A(\r_0-1,\r_0+1),d_M))$) starting at point $x\in \overline{B^{M} (p, R)}$ (resp.\ $x\in A(\r_0-1,\r_0+1)$) remains in $\Omega(\r_0+\bar{R})$.

We first consider the case $x\in \overline{B^{M} (p, R)}$. Let $c$ be a curve in $M$ starting at $x$ of $g$-length $L\leq 2R$. We may extend $c$ to a curve $\tilde{c}$ in $M$ starting at $p\in \Sigma(\r_0)$ of length $L\leq 3R$. Since $M$ is a graph over $\Hb^n$, the $b$-length of the projection of $\tilde{c}$ onto $\Hb^n$ is similarly bounded by $3R$ and this projection starts at $\pr B^b(\r_0)$. So by the triangle inequality the projection must remain in $\overline{B^b(\r_0+3R)}\subset \overline{B^b(\r_0+\bar{R})}$. Hence $\tilde{c}$ itself remains in $\Omega(\r_0+\bar{R})$.

The case $x\in A(\r_0-1,\r_0+1)$ goes similarly: First note that $\diam((A(\r_0-1,\r_0+1),d_M))$ is bounded from above by $4 \sqrt{1+\g^2}+ \pi \sinh(\r_0+1)\sqrt{1+\g^2}$ (this follows as always by considering curves  along the graph of $f$ between any two points in $A(\r_0-1,\r_0+1)$ which first move radially outward, then follow $\Sigma(\r_0+1)$ and then move radially inward). Then we see, analogously to the above, that any curve in $M$ starting at $x$ of $g$-length $L\leq 4 \sqrt{1+\g^2}+ \pi \sinh(\r_0+1)\sqrt{1+\g^2}$ must remain in $ \Omega((\r_0+1)+4 \sqrt{1+\g^2}+ \pi \sinh(\r_0+1)\sqrt{1+\g^2})\subset \Omega(\r_0+\bar{R})$.
\end{proof}

Now we are ready to rewrite Theorem \ref{thm-MainB} and give its proof.  The proof consists in applying Theorem \ref{thm-MainO} and the results of Sections \ref{section:point-convergence} and \ref{ssec-GHannular} in combination with Lemma \ref{lem-distequiv}. 

\begin{thm}
Let $M_j\in\G_n(\r_0,\gamma, D)$ be a sequence of asymptotically hyperbolic graph manifolds with $\lim_{j \to \infty}\m(M_j)=0$
and $p_j \in \Sigma_j(\r_0)$  be a sequence of points.  Then  for almost every $R>0$ we have
\bee
\lim_{j \to \infty} d_{\mathcal{F}}   (   (\overline{B^{M_{j}} (p_{j}, R)} , d_{M_j},  [[ \overline{B^{M_{j}} (p_{j},R) }]])  ,   
( \ovr{B^b(R)} , d_{\Hb^n},  [[ \ovr{B^b(R)} ]] )  )  = 0
\eee
and $\vol(B^{M_j}(p_j,R)) \to \vol(B^b(R))$.
\end{thm}

In the following we abuse notation and do not change indexes when passing to subsequences.

\begin{proof}[Proof of Theorem \ref{thm-MainB}]

Let $\bar{R}>1+(4+\pi \sinh(\r_0+1))\sqrt{1+\g^2}$.  Then by Theorem~\ref{thm-MainO} and Theorem \ref{thm-IFiffF} we know that 
\bee
Q_j:=(\Omega_j (\r_0+ \bar{R}), \di_{\Omega_j (\r_0+ \bar{R})}, [[\Omega_j (\r_0+ \bar{R})]])   \xrightarrow[Z, \varphi_j]{\>\mathcal{F}\>}    (\overline{B^b(\r_0+\bar{R}) }, d_{\Hb^n}, [[\overline{B^b(\r_0+\bar{R}) }]]),
\eee
for some metric space $Z$ and isometric embeddings $\varphi_j$. 
Assume with no loss of generality that $\r_0>2$, so $\r_0 - 1 > \r_0/2$, and $A_j(\r_0-1, \r_0+1) \subset \Omega_j(\r_0+ \bar{R})$ since $\bar R >1$.
Thus  we can apply Corollary~\ref{lem-Phi} to get that 
$$(A_j (\r_0-1, \r_0+1), \di_{\Omega_j (\r_0+ \bar{R})}, [[A_j (\r_0-1, \r_0+1)]])$$ 
converges in intrinsic flat sense to some integral current space
$$(\overline{B^b(\r_0 +1)} \setminus B^b(\r_0-1), d_\infty,  T_\infty)$$
and $(A_j (\r_0-1, \r_0+1), \di_{\Omega_j (\r_0+ \bar{R})})$ converges to $(\overline{B^b(\r_0 +1)} \setminus B^b(\r_0-1), d_\infty)$ in the Gromov-Hausdorff distance. 
Now, notice that  since $d_{M_j}\leq \di_{\Omega_j (\r_0+ \bar{R})}$, we always have $ \overline{B^{\Omega_j(\r_0+\bar R)}(p_j, 1)}\subset \overline{B^{M_j}(p_j, 1)} \subset A_j(\r_0-1, \r_0+1)$. 
So by
 Lemma~\ref{theorem:point-convergenceH} we obtain a subsequence of $p_j \in \Sigma_j(\r_0)$
	such that
	\be\label{eqn:converging-points}
(Q_j, p_j)
	\xrightarrow[Z, \varphi_j]{\>\mathcal{F}\>}   
	\left(    (\overline{B^b(\r_0+\bar{R}) }, d_{\Hb^n}, [[\overline{B^b(\r_0+\bar{R}) }]]), p_\infty\right).
	\ee 
By Lemma  \ref{lem-inner-gone} and Theorem~\ref{theorem:GH}, we get that 
\bee
	d_{\Hb^n}(p_\infty, \partial B^b(\r_0 + \bar{R}) ) \ge
	\limsup_{j\to\infty}    \di_{\Omega_j (\r_0+ \bar{R})}  \left(p_j, \Sigma_j (\r_0 + \bar{R}) \right)\ge  \bar{R}. 
\eee
So $ \overline{ B^b(p_\infty, R) } \subset \overline{B^b(\r_0+\bar{R})}$ holds for almost every $R\in (0, \bar R)$
and using the notation from Example \ref{lem-ball}
\begin{align*}
		S(p_\infty,R) =  & 
		( \overline{ B^b(p_\infty, R) } , d_{\Hb^n}, [[\overline{ B^b(p_\infty, R)  } ]]) \\
	 =	&  ( \overline{ B^b(R) } , d_{\Hb^n}, [[\overline{ B^b(R)  } ]]),
	\end{align*}
	where in the last equality we use the fact that for manifolds $N_i$ with integral current spaces $(N_i,d_{N_i}, [[N_i]])$, we have $(N_1,d_{N_1}, [[N_1]])=(N_2,d_{N_2}, [[N_2]])$ if and only if there exists an orientation preserving isometry between the $N_i$'s. Furthermore, by applying Example \ref{lem-ball} 
	 and Lemma \ref{lem-distequiv} item $(1)$, we get that  for almost every $R \in (0, \frac{\bar{R}}{3})$
\begin{align*}
S(p_j,R)= & \left( \overline{B^{\Omega_j (\r_0+ \bar{R})}(p_j, R) }, \di_{\Omega_j (\r_0+ \bar{R})}, [[\overline{B^{\Omega_j (\r_0+ \bar{R})}(p_j, R)}]] \right) \\
= & (\overline{ B^{M_j}(p_j, R)}  , d_{M_j}, [[  \overline{ B^{M_j}(p_j, R)} ]] ).
\end{align*}

Now by   ~\eqref{eqn:converging-points}, Lemma~\ref{lem-AASorh} and Remark \ref{rmrk-volConv}, we get that for almost every $R>0$ and a subsequence of $p_j$,
  \bee
S(p_j,R)  \xrightarrow[Z, \varphi_{j}]{\>\mathcal{F}\>}    S(p_{\infty},R),
 \eee
and 
$\mass(S(p_j,R))\to \mass(S(p_\infty,R))$. 
Thus,  for a subsequence of $p_j$ and almost every $R\in (0, \frac{\bar{R}}{3})$ it holds 
	\bee
	 (\overline{B^{M_{j}} (p_{j}, R)} , d_{M_j},  [[ \overline{B^{M_{j}} (p_{j},R) }]])  \IFto  (\overline{B^b(R) }, d_{\Hb^n}, [[\overline{B^b(R) }]])
	\eee
	and  $\vol(B^{M_j}(p_{j}, R)) \to \vol(B^b(R)).$

	To finalize the proof, take a sequence of positive real numbers $\bar{R_i} \to \infty$, $\bar R_1>1+(4+\pi \sinh(\r_0+1))\sqrt{1+\g^2}$, and  by a diagonalization argument, proceeding as above, get a further subsequence of the $p_j$ such that for almost all $R>0$ we have $\overline{B^{M_j}(p_{j}, R)} \IFto \overline{B^b(R)}$ and $\vol(B^{M_j}(p_{j}, R)) \to \vol(B^b(R))$. Thus, we can apply Theorem~\ref{theorem:subsequences} to conclude the proof. 
\end{proof}


\appendix

\section{Capping and construction of suitable diffeomorphisms}\label{sec-appendix}

Let $M^n$ be an asymptotically hyperbolic manifold and let $\Psi \colon M \to \Hb^{n+1}$ be an isometric embedding such that $\Psi(M)=\graph(f)$ for some $f \colon \Hb^n \setminus U \to \R$, with non-empty minimal boundary $\pr U$.  Recall that this implies, among other things, that $f$ is constant at the boundary  -- for simplicity assume that $f\vert_{\pr U}=0$. 
 As before, we set $\Omega(\r)=\Psi^{-1}(\ovr{B^b(\r)}  \times \R)$ for $\r >0$. We also identify $\Omega(\rho)$ with $\Psi(\Omega(\rho)) \subset \Hb^{n+1}$  and equip it with the induced metric from $\Hb^{n+1}$.  Our goal is to ``cap-off''  $\Omega(\r)$ to obtain a differentiable manifold $\widetilde{\Omega}$ with a $C^0$ Riemannian metric and to
construct a $C^1$ diffeomorphism $\widetilde{\Phi}:B^b(\r) \to \widetilde{\Omega}$  so that  
the conditions on the Riemannian metrics of Theorem \ref{convBdry} (cf. \cite{A-P,HLP}) are satisfied. More precisely, we are interested in proving the following theorem.
The proof consists in a direct adaptation of the argument in \cite{HLP} to our setting.

\begin{thm}\label{thm-NewOmega}
Let $M\in\G_n(\r_0,\gamma, D)$ be an asymptotically hyperbolic graph with non-empty boundary and
 $\Psi:  M  \to  \Hb^{n+1}$ a smooth Riemannian isometric embedding as in Definition~\ref{def-Gn2}.
For any $\r>\r_0$, $L >0$,  there exists a differentiable manifold
$$\widetilde{\Omega}= \Omega(\r)  \cup \left[ \pr U \times (-L,0)  \right] \cup \widetilde U \subset \Hb^{n+1}$$
 which we endow with the induced Riemannian metric from $(\Hb^{n+1}, \bar b)$, $g$, such that  $g$ is a $C^0$ metric, $\Psi|_{\Omega(\r)}$ is a Riemannian isometric embedding into $\widetilde{\Omega}$ and for any $\lambda \in (0,1)$ there exists a $C^1$ diffeomorphism $\widetilde{\Phi}^\lambda:   \overline{B^b(\r)} \to \widetilde{\Omega}$
with the properties that 
\bee
g(  \widetilde \Phi^\lambda_*(u),  \widetilde \Phi^\lambda_*(u))  \geq \lambda b(u,u),  \quad \forall u \in T \overline{B^b(\r)}.
\eee
Further,  we have upper bounds $\widetilde D(L, \r_0)
$ and  $ \widetilde V(L, \r_0,  \vol(\partial U))
$  for  $\diam(\widetilde{\Omega}  \setminus \Omega(\r), d_{\widetilde{\Omega}}) $ and $\vol( \widetilde{\Omega} \setminus \Omega(\r))$, respectively, given by
\bee
\widetilde D(L, \r_0): =\diam( \widetilde{\Omega}  \setminus \Omega(\r),\di_{\widetilde{\Omega}\setminus \Omega(\r)}) \leq   2\cosh(\tfrac{\rho_0}{2}) L +   \tfrac{1}{2}+\cosh(\tfrac{\rho_0}{2})+ \rho_0 
\eee 
and 
\bee 
\widetilde V(L, \r_0,  \vol(\partial U)) := \Big(L \cosh(\tfrac{\rho_0}{2})+ \tfrac{1}{n-1}+2\cosh(\tfrac{\rho_0}{2})\Big)\, \vol(\partial U)
\eee 
and, for $L> D+\tfrac{1}{2}\sinh(\rho_0)\,\pi \sqrt{1+\gamma^2}$, the space $(\Omega(\r), \di_{\Omega(\r)})$ embeds in a distance preserving way into $\widetilde{\Omega}$.
\end{thm}

The remainder of the appendix will be dedicated to working \emph{in detail} through all the steps necessary in the construction from \cite{HLP} in our hyperbolic setting.

\subsection{Normal exponential map, defining $\widetilde U$ and $f_c$}\label{subsec:cappingFunc}

Consider the normal exponential map of $\partial U \subset \Hb^n$. By compactness of $U$, there exists an $\epsilon_*>0$ and an open neighborhood $N_{\epsilon_*}$ of $\partial U$ such that $E : (-\epsilon_*,\epsilon_*) \times \partial U \to N_{\epsilon_*}, (t,x)\mapsto \exp_x(t\,\nu(x))$  is a diffeomorphism and by shrinking $\epsilon_*$ we may assume that $E$ is still a diffeomorphism on $I\times \partial U$ for some open subset $I$ with $[-\epsilon_*,\epsilon_*]\subset I$.

With the previous notation, we are able to define a map 
 $\phi$ from $(-L,\epsilon_*) \times \partial U$ to $\Hb^{n+1}$ with image in $\widetilde{\Omega}\subset \Hb^{n+1}$ by setting  $\phi(t,x):=(E(t,x), f(E(t,x)))$ if $t\geq 0$ and $\phi(t,x):=(E(0,x),t)$ if $t\leq 0$. Note that $\phi$ is not a $C^1$ map (but piecewise smooth),  so we will replace $\phi$ by a $C^1$ map $\Phi^\lambda$ in order to construct $\widetilde{\Phi}^\lambda$ appearing in 
 Theorem \ref{thm-NewOmega}.  This will be done by introducing an appropriate scaling. The scaling will crucially rely on $\nabla^b_{\nu}f=\tfrac{d}{dt} f(E(t,x))\to \infty$ as $t\to 0^+$. This will show that  the image of $\Phi^\lambda$ is indeed a $C^1$ submanifold of $\Hb^{n+1}$.

Now we want to construct a cap $\widetilde{U}$ which will be given as the graph of a capping function $f_c: \overline{U} \to [-1,0]$, smooth on $U$ and satisfying $f_c|_{\pr U}=0$ and $\nabla^b_{\nu}f_c=\tfrac{d}{dt} f_c(E(t,x))\to \infty$ as $t\to 0^-$, i.e.\ as one approaches $\partial U$, so that we will be able to attach it to the cylinder and will be able to make the resulting manifold $C^1$ in the same way as we will be doing for $f$ above. We construct this capping now: 

We start by defining $f_c $ on $ E((-\epsilon_*,0] \times \partial U)\subset \overline{U}\subset \Hb^n$ by setting
$$ f_c(E(t,x)):=\chi(t)$$
where $\chi : (-\epsilon_*,0]\to [-1,0]$ is smooth on $(-\epsilon_*,0)$, strictly monotonically increasing on $(-\tfrac{\epsilon_*}{2},0]$, $\chi(0)=0$, $\chi|_{(-\epsilon_*,-\tfrac{\epsilon_*}{2}]}=-1$ and $\chi'(t)\to \infty $ as $t\to 0^-$. On $\overline U \setminus E((-\epsilon_*,0] \times \partial U)$ we simply set $f_c= -1$.

By shrinking $\epsilon_*$ if necessary, we further assume that for all $s\in (-1,0]$ we have $\H^{n-1}((f_c\circ E)^{-1}(s))=\vol_{\omega_{\chi^{-1}(s)}}(\pr U)\leq 2\vol_{\omega_0}(\pr U) =2|\pr U|_b $,  where $\{\omega_t\}$ is the family of Riemannian metrics on $\partial U$ induced by the normal exponential map (see \eqref{eq:n-metricinnormalexp} below). 
Recall that  $\pr U$ is star-shaped and this implies that $\pr U$ has a single connected component which is $C^1$-diffeomorphic to $\bS^{n-1}$ via the graph $\theta \mapsto (\rho_{\pr U}(\theta),\theta)\in (0,\infty)\times \bS^{n-1} \cong \Hb^n\setminus \{0\} $. Thus, we may further shrink $\epsilon_*$ to additionally ensure that the level sets  are star-shaped as well: Since, by assumption, $\r_{\pr U}:\bS^{n-1}\to (0,\infty)$ is a $C^1$ function whose graph coincides with $\pr U$, the spherical coordinate vector field $\pr_{\r}$ can never be tangential to $\pr U$ and $\langle \pr_{\r},\nu \rangle_b>0$ near $\pr U$ and one can use the monotonicity of $t\mapsto \rho(E(t,\theta))$ and an implicit function theorem argument to obtain a unique differentiable family of differentiable maps $\rho_d:\bS^{n-1}\to (0,\infty)$, $d\in [-\tfrac{\epsilon_*}{2},0]$, satisfying $\{(\r_d(\theta),\theta)\in \Hb^n:\,\theta \in \bS^{n-1}\}=E(\{d\}\times \pr U)$. This construction is very similar to the process we will encounter in defining the scaling $\tilde{\alpha}$ in Lemma \ref{lemma-alpha-x} hence we will skip the details here.

\subsection{Construction of $\widetilde \Phi^\lambda$}

We already introduced the normal exponential map of $\partial U$, which gives a diffeomorphism $E : (-\epsilon_*,\epsilon_*) \times \partial U \to N_{\epsilon_*}$ onto an open neighborhood $N_{\epsilon_*}$ of $\partial U$ in $ \Hb^n$ for $\epsilon_*>0$ as before.  Pulling back the metric and using "t" to denote the coordinate on $(-\epsilon_*,\epsilon_*)$, standard properties of the normal exponential map give that the metric splits as
\be\label{eq:n-metricinnormalexp} 
E^*b=dt^2+\omega_t 
\ee
where $\omega_t=\omega_{ij}(t,x) dx^i dx^j$ (for coordinates $(x^i)$ on $\partial U$) is a family of Riemannian metrics on $\partial U$.
Let  $\hat{f}:= f\circ E: (0, \epsilon_*)\times \partial U\to\Omega(\rho)$ and $\hat{f}_c:= f_c\circ E: (-\epsilon_*, 0)\times \partial U\to \Hb^{n+1}$ be the expressions of $f$ and $f_c$ in the geodesic normal coordinates. 

Fix $\lambda \in (0,1)$.  Let $\veps\equiv \veps(\lambda, f,\partial U, L, f_c) \in (0,1)$, $\veps<\epsilon_*$, satisfy  $\veps <\veps_0$ for $\veps_0$ from the Lemmas \ref{lem:alpha-props} and \ref{lem:beta}, $\tfrac{2L}{\veps}\geq 1$ 
and
\begin{align}\label{eq:vepscondition}
\inf_{(t,s)\in (-\veps,\veps)\times (-\veps,\veps)} \;  \min_{x\in \partial U, \,\bar{u}\in \R^{n-1}, |\bar{u}|_e=1} \tfrac{  \sum_{i,j}\omega_{ij}(s,x)\bar{u}^i \bar{u}^j}{ \sum_{i,j} \omega_{ij}(t,x)\bar{u}^i \bar{u}^j }\geq \sqrt{\lambda}>\lambda 
\end{align}
where $\| \cdot \|_e$ denotes the Euclidean norm on $\R^{n-1}$. Note that this is possible because the function in the infimum above is continuous and equal to one on the diagonal. 

We will define a $C^1$ diffeomorphism $\Phi^\lambda$ from $(-\veps,\veps)\times \pr U$ to an open subset of $\widetilde{\Omega}$ containing $\partial U\times (-L,0)$ such that  $\Phi^\lambda|_{(\tfrac{2 \veps}{3},\veps )} = (\mathrm{id}, \hat{f})$,  $\Phi^\lambda|_{(-\veps, -\tfrac{5 \veps}{6} )} = (\mathrm{id}, \hat{f}_c-L)$ 
and  $g(\Phi^\lambda_*(u),\Phi^\lambda_*(u))\geq \lambda \, (E^*b)(u,u)$ holds for all $u\in T((-\veps,\veps)\times \pr U))$. 

Once we have defined $\Phi^\lambda$ we will set
$$\widetilde{\Phi}^\lambda(p) :=\begin{cases}
	(p, f_c(p)-L) \quad \quad \quad p\in U\setminus E((-\veps,\veps)\times \partial U)\\
	\Phi^\lambda(E^{-1}(p)) \quad \quad p\in
	E((-\veps,\veps)\times \partial U)\\
	(p, f(p)) \quad \quad \quad p\in \Omega(\rho)\setminus E((-\veps,\veps)\times \partial U)
\end{cases}
$$
giving us the desired $C^1$ diffeomorphism from $\overline{B^b(\rho)}\to \widetilde{\Omega}$ satisfying  $g(\widetilde{\Phi}^\lambda_*(u),\widetilde{\Phi}^\lambda_*(u))\geq \lambda \, b(u,u)$  everywhere because it is either a graph (where $g\geq b$ automatically since $V\geq 1$) 
 or given by $\Phi^\lambda \circ E^{-1}$ which will be constructed to do so.

\subsubsection{Definition of $\Phi^{\lambda}$ on $(0,\veps)\times \pr U$} 

To make $\Phi^\lambda$ continuously differentiable across the gluing surface between $\Omega(\rho)$ and the cylinder, we re-scale the $t$ parameter. This is accomplished by the function $\alpha$ from Lemma ~\ref{lem:alpha-props}: 
Set 
$$\Phi^\lambda(t,x):= (E(\alpha(t,x), x), \hat{f}(\alpha(t,x),x))\in \Omega(\rho) \subset \Hb^{n+1} \quad\quad \mathrm{for}\; (t,x)\in (0,\veps)\times \partial U.$$
By the properties of $\alpha$ this smoothly matches $(\mathrm{id}, \hat{f})$ near $t=\veps$. We now show that 
$g(\Phi^\lambda_*(u),\Phi^\lambda_*(u))\geq \lambda \, (E^*b)(u,u)$ holds for all $u\in T((0,\veps)\times \pr U))$. 

For bookkeeping let us denote local coordinates on  $(0,\veps)\times \partial U$ by $t$ and $(x^i)$. These induce canonical coordinates on the graph $\{(E(t,x), \hat{f}(t,x))\in \Hb^{n+1}: (t,x)\in (0,\veps) \times \partial U \}\subset \mathrm{Image}(\Phi^\lambda)\subset \Omega(\rho) \hookrightarrow \Hb^{n+1}$, which we - by a  slight abuse of notation - again denote by $t$ and $(x^i)$, and on $E((0,\veps)\times \partial U) \times \R \subset \Hb^{n+1}$, which we denote by $t,(x^i),s$.

In the coordinates we are using, 
\begin{equation}\label{eqn:phigraph}
\Phi^\lambda(t,x):= (\alpha(t,x), x, \hat{f}(\alpha(t,x),x))\in  (0,\veps)\times \partial U	\times \R
\end{equation}
and we have:
\begin{align*}
& \Phi^\lambda_*(\partial_t|_{(t,x)}) =  (\partial_t\alpha)|_{(t,x)}\, \partial_t|_{\Phi^\lambda(t,x)} +   (\partial_t \hat{f})|_{(\alpha(t,x),x)} \,(\partial_t\alpha)|_{(t,x)}\, \partial_s|_{\Phi^\lambda(t,x)}\\
& \Phi^\lambda_*(\partial_{x^i}|_{(t,x)}) =   (\partial_i \alpha)|_{(t,x)} \partial_{t}|_{\Phi^\lambda(t,x)}+\partial_{i}|_{\Phi^\lambda(t,x)} + \big[(\partial_t \hat{f})|_{(\alpha(t,x),x)} \,(\partial_i\alpha)|_{(t,x)}+(\partial_i \hat{f})|_{(\alpha(t,x),x)}\big]\, \partial_s|_{\Phi^\lambda(t,x)}.
\end{align*}
Since the metric $g$ on $\widetilde \Omega$ is induced by $\bar{b}$, we have
\begin{align*}
g(\Phi^\lambda_*(\partial_t|_{(t,x)}),\Phi^\lambda_*(\partial_t|_{(t,x)}))  = &\bar{b}(\Phi^\lambda_*(\partial_t|_{(t,x)}),\Phi^\lambda_*(\partial_t|_{(t,x)}))\\
	 = & (\partial_t\alpha)^2|_{(t,x)} + V(r(\a(t,x),x))^2  (\partial_t \hat{f})^2|_{(\alpha(t,x),x)} \,(\partial_t\alpha)^2|_{(t,x)} \\
g(\Phi^\lambda_*(\partial_t),\Phi^\lambda_*(\partial_i))  = & \partial_t\alpha \; \partial_i\alpha + V^2 \big[ \partial_t\hat{f}\,\partial_i\alpha +\partial_i \hat{f} \big] \,\partial_t \hat f\; \partial_t\alpha \\
g(\Phi^\lambda_*(\partial_i),\Phi^\lambda_*(\partial_j))  = & (\partial_i\alpha)(\partial_j\alpha) +\omega_{ij} + V^2 \big[ \partial_t\hat{f}\,\partial_i\alpha +\partial_i \hat{f}  \big] \big[ \partial_t \hat{f}\,\partial_j\alpha +\partial_j \hat{f}  \big],
	\end{align*}
where we suppressed the arguments for the second and third formula (as a rule of thumb: $\hat{f}$ and its derivatives will have argument $(\alpha(t,x),x)$, $\alpha$ and its derivatives $(t,x)$ and $\omega$  and $V=\cosh$ are evaluated at $\mathrm{pr}_{\Hb^n}(\Phi^\lambda(t,x))=(\a(t,x),x)$ and $r(\a(t,x),x)$ respectively) and used the form of $\bar{b}=b + V(r)^2ds^2 = dt^2+\omega_t + V(r)^2ds^2 = dt^2+\omega_{ij}(t,x) dx^idx^j + V(r)^2ds^2$ in our chosen coordinates. 

We introduce the following notation
\begin{align} 
 B^j(t,x) &:= \omega^{ij}(\a(t,x),x)  \big[(\partial_t \hat{f})|_{(\alpha(t,x),x)} \,(\partial_i\alpha)|_{(t,x)}+(\partial_i \hat{f})|_{(\alpha(t,x),x)}\big],  \label{eq:defB} \\
	A^j(t,x) &:= \omega^{ij}(\a(t,x),x)   (\partial_i\alpha)(t,x).  \label{eq:defA}
\end{align}
With this the above expressions can be rewritten as
\begin{align*}
g(\Phi^\lambda_*(\partial_t),\Phi^\lambda_*(\partial_t)) & =(\partial_t\alpha)^2 + V^2  (\partial_t \hat{f})^2 \,(\partial_t\alpha)^2 \\
g(\Phi^\lambda_*(\partial_t),\Phi^\lambda_*(\partial_i))&= \partial_t\alpha \; \omega_{ik} A^k + V^2 \, \omega_{ik} B^k\, \partial_t \hat{f} \, \partial_t\alpha\\
g(\Phi^\lambda_*(\partial_i),\Phi^\lambda_*(\partial_j))&=\omega_{jk} A^k  \omega_{i\ell}A^{\ell} +\omega_{ij} + V^2 \, \omega_{jk} B^k  \omega_{i\ell}B^{\ell}.
\end{align*}
Now let $u=a\partial_t|_{(t,x)} +u^i \partial_i|_{(t,x)}=a\partial_t|_{(t,x)} +\bar{u} \in T_{(t,x)}((0,\veps)\times \pr U)$ be arbitrary and compute (ignoring basepoints/arguments for now for readability)
\begin{align}\nonumber
	g(\Phi^\lambda_*(u),\Phi^\lambda_*(u))  &= a^2 g(\Phi^\lambda_*(\pr_t),\Phi^\lambda_*(\pr_t))+  2au^i g(\Phi^\lambda_*(\pr_t),\Phi^\lambda_*(\pr_i)) +           u^iu^jg(\Phi^\lambda_*(\pr_i),\Phi^\lambda_*(\pr_j))  \\ \nonumber
	&=  a^2 (\partial_t\alpha)^2 + a^2 V^2  (\partial_t \hat{f})^2 \,(\partial_t\alpha)^2 + 2 a \langle \bar{u},A\rangle_{\omega} \,\partial_t\alpha + 2 a V^2 \langle \bar{u},B\rangle_{\omega} \partial_t \hat{f} \,\partial_t\alpha \\
	&+  \langle\bar{u},A\rangle_{\omega}^2+ \omega_{(\a(t,x),x)}(\bar{u},\bar{u}) + V^2 \langle \bar{u},B\rangle_{\omega}^2 \label{eqn:gest1}.
\end{align}
By Lemma \ref{lem:alpha-props} we know that either 
\begin{align*}
& (\partial_t \hat{f})^2 \,(\partial_t\alpha)^2\geq 1 \quad \text{and} \quad |B(t,x)|_{\omega(\a(t,x),x)}\leq \tfrac{1}{V(r(\a(t,x),x))^2} (1-\sqrt{\lambda}) \quad \text{or}, \\
&  (\partial_t\alpha)^2\geq 1 \quad  \text{and} \quad |A(t,x)|_{\omega(\a(t,x),x)}\leq (1-\sqrt{\lambda}).
  \end{align*}
In the first case we estimate
\begin{align}\nonumber
	g(\Phi^\lambda_*(u),\Phi^\lambda_*(u))  & \geq  (a \partial_t\alpha+ \langle\bar{u},A\rangle_{\omega})^2+ a^2 V^2  (\partial_t \hat{f})^2 \,(\partial_t\alpha)^2+ 2 a V^2 \langle \bar{u},B\rangle_{\omega} \partial_t \hat{f} \,\partial_t\alpha  + \omega_{(\a(t,x),x)}(\bar{u},\bar{u}) \\\nonumber
	&\geq a^2 V^2  (\partial_t \hat{f})^2 \,(\partial_t\alpha)^2 - V^2 |B|_{\omega(\a(t,x),x)} (2 a  |\partial_t \hat{f}| \,|\partial_t\alpha| |\bar{u}|_{\omega(\a(t,x),x)}) + \omega_{(\a(t,x),x)}(\bar{u},\bar{u})   \\\nonumber
	&\geq a^2 V^2  (\partial_t \hat{f})^2 \,(\partial_t\alpha)^2 (1-|B|_{\omega(\a(t,x),x)}) +  \omega_{(\a(t,x),x)}(\bar{u},\bar{u}) (1-V^2|B|_{\omega(\a(t,x),x)})\\\nonumber
	&\geq a^2 \sqrt{\lambda} + \omega_{(\a(t,x),x)}(\bar{u},\bar{u}) \sqrt{\lambda} \geq a^2 \lambda + \omega_{(\a(t,x),x)}(\bar{u},\bar{u}) \sqrt{\lambda} 
\end{align} 
where we used the Cauchy-Schwarz inequality and that $V^2 \geq 1$. Now by \eqref{eq:vepscondition} we have $\omega_{(\a(t,x),x)}(\bar{u},\bar{u})\geq \sqrt{\lambda} \omega_{(t,x)}(\bar{u},\bar{u})$ (note that this is scaling invariant so it is sufficient to consider $\bar{u}$ with Euclidean norm equal one) and we obtain the desired
\begin{align}
	g(\Phi^\lambda_*(u),\Phi^\lambda_*(u))  &\geq  \lambda  (a^2 + \omega_{(t,x)}(\bar{u},\bar{u}))= \lambda b(u,u).
\end{align} 

In the second case we proceed similarly and estimate
\begin{align}\nonumber
	g(\Phi^\lambda_*(u),\Phi^\lambda_*(u))  & \geq a^2 (\partial_t\alpha)^2 + (a V  \partial_t \hat{f} \,\partial_t\alpha + V \langle \bar{u},B\rangle_{\omega})^2 + 2 a  \langle \bar{u},A\rangle_{\omega} \,\partial_t\alpha  + \omega_{(\a(t,x),x)}(\bar{u},\bar{u}) \\\nonumber
	&\geq  a^2 (\partial_t\alpha)^2 -  |A|_{\omega(\a(t,x),x)} (2 a \,|\partial_t\alpha| |\bar{u}|_{\omega(\a(t,x),x)})  + \omega_{(\a(t,x),x)}(\bar{u},\bar{u})  \\\nonumber
	&\geq a^2 (\partial_t\alpha)^2 (1-|A|_{\omega(\a(t,x),x)}) +  \omega_{(\a(t,x),x)}(\bar{u},\bar{u}) (1-|A|_{\omega(\a(t,x),x)})\\\nonumber
	&\geq a^2 \sqrt{\lambda} + \omega_{(\a(t,x),x)}(\bar{u},\bar{u}) \sqrt{\lambda} \geq  \lambda b(u,u)
\end{align} 
again using the Cauchy-Schwarz inequality and \eqref{eq:vepscondition}.

\subsubsection{Definition of $\Phi^\lambda$ on $(-\tfrac{\veps}{2},0)\times \partial U$}

We set 
\begin{equation}
	\label{eqn:phicyl} \Phi^\lambda(t,x):= (E(0,x), \tfrac{2L}{\veps}t)\equiv (0, x, \tfrac{2L}{\veps}t ) \quad\quad\mathrm{for}\;(t,x)\in (-\tfrac{\veps}{2},0)\times \partial U.
\end{equation}
Let us check that $\Phi^\lambda :(-\tfrac{\veps}{2},\veps)\to \widetilde{\Omega}$ defined by \eqref{eqn:phigraph} and \eqref{eqn:phicyl} (and continuously extended across $t=0$ where both functions become $(0,x, 0)$) is $C^1$: On $(0, \veps) \times \partial U$ from \eqref{eqn:phigraph}
we see 
\begin{align*}
	(\partial_t \Phi^\lambda)(t,x) &= (\pr_t\a|_{(t,x)} ,0, \pr_t \hat f|_{(\alpha(t,x),x)} \,\pr_t\alpha|_{(t,x)})   \to (0,0, \tfrac{2L}{\veps})\quad \mathrm{as}\;t\to 0^+\\
		(\partial_i \Phi^\lambda)(t,x) &= (\pr_i\a|_{(t,x)},e_i, \partial_t \hat{f}|_{(\alpha(t,x),x)} \,\partial_i\alpha|_{(t,x)}+\partial_i \hat{f}|_{(\alpha(t,x),x)}) \to (0,e_i, 0)\quad \mathrm{as}\;t\to 0^+
\end{align*}
by Lemma \ref{lem:alpha-props}. On $(-\tfrac{\veps}{2},0)\times \partial U$ clearly $(\partial_t \Phi^\lambda)(t,x) = (0,0, \tfrac{2L}{\veps})$ and $(\partial_i \Phi^\lambda)(t,x) = (0,e_i,0)$, so $\Phi^\lambda$ is indeed $C^1$ across $\veps=0$.

Next we check that $g(\Phi_*^\lambda (u),\Phi_*^\lambda (u))\geq \lambda (E^*b)(u,u)$ for $u=a\pr_t+u^i\pr_i\in T_{ (t,x)} (-\tfrac{\veps}{2},0] \times \partial U$ for $(t,x)\in (-\tfrac{\veps}{2},0]\times \pr U$: Using $\Phi^\lambda_*(\pr_t)=\tfrac{2L}{\veps} \pr_s$ and $\Phi^\lambda_*(\pr_i)= \pr_i$ we get
 \begin{equation}
 	g(\Phi_*^\lambda (u),\Phi_*^\lambda (u))= \omega_{ij}(0,x) u^i u^j  + a^2 V^2 (\tfrac{2L}{\veps})^2 \geq  \omega_{ij}(0,x) u^i u^j + a^2 \geq \lambda (E^*b)(u,u)
 \end{equation}
where we used that $\tfrac{2L}{\veps}\geq 1$ and \eqref{eq:vepscondition} by assumption on $\veps$.

\subsubsection{Definition of $\Phi^\lambda$ on $(-\veps,-\tfrac{\veps}{2})\times \partial U$} 

This will be completely analogous to the first part, but using $\hat{f}_c-L$ instead of $\hat{f}$ and the function $\alpha_c : (-\veps, -\tfrac{\veps}{2})\to (-\veps,0)$ from Lemma \ref{lem:beta} instead of $\alpha$. We set
\begin{equation}\label{eqn:phicap}
 \Phi^\lambda(t,x):= ( E(\a_c(t,x),x), \hat{f}_c(\a_c(t,x),x)-L)\equiv (\a_c(t,x),x, \hat{f}_c(\a_c(t,x),x)-L)
\end{equation}
for $(t,x)\in (-\veps,-\tfrac{\veps}{2})\times \partial U$.
Similarly to the computations for $\Phi^\lambda$ on $(0,\veps) \times \partial U$ and  $(-\tfrac{\veps}{2},0) \times \partial U$ we see that this has the desired properties.

\subsection{Defining the rescalings $\alpha$, $\alpha_c$} 

We first show that we can guarantee suitable behavior near $0$. 

\begin{lemma} \label{lemma-alpha-x} 
Fix $C\in (0,\infty)$, $c\in (0,\epsilon_*)$. There exists $0<\delta_0(C, c, \hat{f})<\tfrac{c}{3}$ 
and a smooth, strictly increasing in $t$, map $\ta \colon (0,\delta_0)\times \pr U \to (0, \tfrac{c}{3})$ such that  for all $x\in \pr U$
\begin{enumerate}
	\item $\hat{f}(\ta(t,x),x)=C\,t$ for $\,  t \in (0,\delta_0)$ and
\item $\lm{t}{0^+}\ta(t,x)=0$ 
\end{enumerate}
\end{lemma}

\begin{proof} 
Consider the map $F(t,x):= \tfrac{1}{C}\hat{f}(t,x)$. 
This is a smooth map from $(0,\epsilon_*)$ to $(0,\infty)$ which extends continuously to $0$ as one approaches $\{0\}\times \pr U$. By the asymptotic assumption on $\nabla_\nu^b f$ we have
$\lim_{s\to 0^-} (\pr_t\hat{f})(s,x)=\infty $ for all $x\in \pr U$, so we may choose $\veps(C,c,\hat{f}, x)\equiv \veps(x)<\tfrac{c}{3}$ small enough such that $s\mapsto F(s,x)$ is strictly monotonically increasing on $(0,\veps(x)]$. Note that we may choose $\veps(x)$ depending continuously on $x$. 
    Let $\delta_0(C,c,\hat{f}):= \min_{x\in \pr U} F(\veps(x), x)$ and set $\tilde{\veps}(x):= F^{-1}(\delta_0,x)\leq \veps(x)<\tfrac{c}{3}$. Then $s\mapsto F(s,x)$ maps $(0,\tilde{\veps}(x))$ bijectively onto $(0,\delta_0)$, so for any $t< \delta_0$  there exists a {\em unique} solution $\ta(t,x)\in [0,\tilde{\veps}(x))$ to the equation $\tfrac{1}{C} \hat{f}(\ta(t,x),x)=t$. Clearly $\ta \colon (0,\delta_0)\times \pr U \to (0, \tfrac{c}{3})$ satisfies $\lm{t}{0^+}\ta(t,x)=0$ for all $x\in \pr U$ an is strictly increasing in $t$. It remains to argue that $(t,x)\mapsto \ta(t,x)$ is smooth. For this we note that $\ta$ will agree with maps obtained from the local implicit function theorem near any $(t_0,x_0)\in (0,\veps_0)\times \pr U$: Since $G(\alpha,t,x):=F(\alpha, x)-t$ satisfies $G(\alpha_0,t_0,x_0)=0$ and $\pr_\alpha G|_{(\a_0,t_0,x_0)}>0$ for $\alpha_0:=\ta(t_0,x_0)\in (0,\tfrac{c}{3})$ there exists locally around $(t_0,x_0)$ a smooth map $(t,x)\mapsto \a(t,x)$ to a neighborhood of $\alpha_0\in (0,\tilde{\veps}(x_0))$ with $G(\a(t,x),t,x)=0$. So by the noted uniqueness of the solution $\ta(t,x)$ to $\tfrac{1}{C} \hat{f}(\ta(t,x),x)=t$, $\ta=\a$ on this neighborhood and smoothness of $\ta$ follows.    
\end{proof}

We now obtain $\a$ by interpolating between $\ta$ near $t=0$ and the identity near $t=\veps$.

\begin{lemma}\label{lem:alpha-props}
Fix $L>0$ and $0<\lambda <1$. There exists $\veps_0(L,\lambda,\pr U, \hat{f})>0$ such that for any $0<\veps<\veps_0$  there exists a $C^1$ function $\a: (0,\veps)\times \pr U\to (0,\veps)$ (depending on $\hat{f}$, $\veps$, $L$ and $\lambda$) 
such that
\begin{enumerate}
	\item for any $x\in \pr U$, $t\mapsto \a(t,x)$ is strictly increasing,
	\item $\hat{f}(\a(t,x),x)=\tfrac{2L}{\veps}\,t$ for $t$ near $0$,
	\item $\a(t,x)\to 0$ as $t\to 0^+$ and
	\item $\a(t,x)=t$ for $(t,x)\in (\tfrac{2\veps}{3},\veps)\times \pr U$.
\end{enumerate}
Note that (2) implies the following formulas for the derivatives  near $t=0$
\begin{align}
	 (\pr_t\hat{f})|_{(\a(t,x),x)}\;(\pr_t\a)|_{(t,x)}&=\tfrac{2L}{\veps}\\
	 (\pr_t\hat{f})|_{(\a(t,x),x)}\;(\pr_i\a)|_{(t,x)}+ (\pr_i\hat{f})|_{(\a(t,x),x)}&= 0.
\end{align}
In particular  $(\pr_t\a)(t,x), (\pr_i\a)(t,x) \to 0$ as $t\to 0^+$. 
Further, for any $(t,x)\in (0,\veps)\times \pr U$ we can guarantee either 
\begin{itemize}
	\item $\partial_t \hat{f}|_{(\a(t,x),x)} \,(\partial_t\alpha)|_{(t,x)}\geq 1$ and $|B(t,x)|_{\omega(\a(t,x),x)}\leq \tfrac{1}{V(r(\a(t,x),x))^2} (1-\sqrt{\lambda})$ or
\item 	$\partial_t\alpha|_{(t,x)}\geq 1$ and $|A(t,x)|_{\omega(\a(t,x),x)}\leq 1-\sqrt{\lambda}$
\end{itemize} 
where $B(t,x), A(t,x)$ are as in \eqref{eq:defB} resp.\ \eqref{eq:defA}.
\end{lemma}

\begin{proof} 
Choose $\veps_0(L,\lambda,\pr U, \hat{f})<\min(\epsilon_*,2L)$ such that  $\hat{f}$ satisfies $\pr_t\hat{f}>1$ on $(0,\veps_0)\times \pr U$ and  \begin{align}\label{eq:Aesttobe}
		\Big|\omega^{ij}(s,x) \tfrac{1}{(\pr_t\hat{f})|_{(t,x)}} (\pr_i\hat{f})|_{(t,x)}\Big|_{\omega(s,x)}< 1-\sqrt{\lambda}
	\end{align} for any $(s,t,x)\in (0,\veps_0)\times (0,\veps_0)\times \pr U$. 
	Fix any $0<\veps<\veps_0$. Then, by Lemma \ref{lemma-alpha-x}  there exists a $\delta_0(\tfrac{2L}{\veps},\veps)$ and a $\ta : (0,\delta_0)\times \pr U \to (0,\tfrac{\veps}{3})$ which is smooth, satisfies (2) and is strictly increasing in $t$.
	
	For any $\xi \in (0,\min(\delta_0,\tfrac{\veps}{3}))$ small enough such that $(\pr_t\ta)(t,x)=\tfrac{2L}{\veps}(\pr_t\hat{f}(\ta(t,x),x))^{-1})<1$ for all $(t,x)\in (0,\xi)\times \pr U$  we set
$$		\a(t,x)=\begin{cases*} \ta(t,x) \quad\quad\quad\quad \quad\quad\quad\quad\quad\quad \quad \mathrm{on} \;\; (0,\xi]\times \pr U \\
			\tfrac{\veps-2\ta(\xi,x)}{\veps-2\xi}\,(t-\xi)+\ta(\xi,x) \quad\quad \mathrm{on} \;\;  [\xi,\tfrac{\veps}{2}]\times \pr U \\
		 t \quad\quad\quad\quad\quad\quad\quad\quad\quad\quad\quad\quad\quad\; \mathrm{on} \;\;  [\tfrac{\veps}{2},\veps)\times \pr U,
	\end{cases*} $$
	where the middle segment is a simple linear interpolation. This is piecewise $C^1$ and smooth in $x$ for any fixed $t$ with $(t,x)\mapsto (\pr_i\a)(t,x)$ continuous. Further, the slope of the linear segments is $\geq 1$ by construction since $\ta(\xi,x)<\xi$. Assumption \eqref{eq:Aesttobe} ensures $|A(t,x)|_{\omega(\a(t,x),x)}< 1-\sqrt{\lambda}$ for all $(t,x)\in (0,\veps)\times \pr U$. (We will skip the detailed computations here, let us however remark that we get that  $|A(t,x)|_{\omega}$ is given by the left hand side of \eqref{eq:Aesttobe} for $t\in (0,\xi)$, while on the first linear segment $|A(t,x)|_{\omega}$ is essentially given by \eqref{eq:Aesttobe}$|_{t=\xi}\cdot \tfrac{\veps-2t}{\veps-2\xi} $, 
		and on the second linear segment $A(t,x)=0$.) On $(0,\xi]$ on the other hand, $B(t,x)=0$ (note that $B$ is continuous on $(0,\veps)$) and $\partial_t \hat{f}|_{(\a(t,x),x)} \,(\partial_t\alpha)|_{(t,x)}=\tfrac{2L}{\veps}> 1$. So this $\a$ satisfies all our requirements except only being piecewise $C^1$. To make it $C^1$ we only need to modify it in arbitrarily small neighborhoods of the corners at $\{\xi\}\times \pr U$ and $\{\tfrac{\veps}{2}\}\times \pr U$ (and join the modified bits to the original $\alpha$ using a partition of unity). At both corners the $C^1$ versions of $\a$ can be created via convolution in $t$ while fixing $x$. At the second corner this clearly preserves  $\pr_t\a \geq 1$ and $|A(t,x)|_{\omega(\a(t,x),x)}< 1-\sqrt{\lambda}$. For the second corner note that $\hat{f},\omega,V$ are smooth near $\{\xi \}\times \pr U$ and $\pr_i$ commutes with the smoothing of $\a$ in $t$, so we can use $B=0$ on $\{\xi\}\times \pr U$ to keep $V^2\,|B|_{\omega}$ arbitrarily small, in particular smaller than $1-\sqrt{\lambda}$, by making the neighborhood of $\{\xi\}\times \pr U$ on which we modify small enough. Secondly, since we have $\pr_t\hat{f}>1$ on $(0,\veps)$ and $\pr_t\a \geq 1$ for $t>\xi$ it follows that $\partial_t \hat{f}|_{(\a(t,x),x)} \,(\partial_t\alpha)|_{(t,x)}>1$ on $(0,\xi)$ and $(\xi,\veps)$, so we may use smoothness of $\hat{f}$ near $\{\xi\}\times \pr U$ and boundedness of $\pr_t\a$ to ensure that also the smoothed $\alpha$ satisfies $\partial_t \hat{f}|_{(\a(t,x),x)} \,(\partial_t\alpha)|_{(t,x)}>1$ provided we make the neighborhood of $\{\xi\}\times \pr U$ on which we modify small enough.  
\end{proof}

\begin{lemma}\label{lem:beta} 
Fix $L>0$, $0<\lambda <1$ and the capping function $f_c$ constructed in Section \ref{subsec:cappingFunc}. There exists $\veps_0(L,\lambda,\pr U, \hat{f}_c)>0$ such that for any $0<\veps<\veps_0$  there exists a $C^1$ function $\a_c: (-\veps,-\tfrac{\veps}{2}) \times \pr U\to (-\veps, 0)$ (depending on $\hat{f}_c$, $\veps$, $L$ and $\lambda$) 
	such that
		\begin{enumerate}
			\item for any $x\in \pr U$, $t\mapsto \a_c(t,x)$ is strictly increasing,
			\item $\hat{f}_c(\a_c(t,x),x)=\tfrac{2L}{\veps}\,(t+\tfrac{\veps}{2})$ for $t$ near $-\tfrac{\veps}{2}$,
			\item $\a_c(t,x)\to 0$ as $t\to -\tfrac{\veps}{2}^-$,
			\item $\a_c(t,x)=t$ for $(t,x)\in (-\veps,-\tfrac{5\veps}{6})\times \pr U$
		\end{enumerate}
		Note that (2) implies the following formulas for the derivatives near $t=-\tfrac{\veps}{2}$
		\begin{align}
			(\pr_t\hat{f}_c)|_{(\a_c(t,x),x)}\;(\pr_t\a_c)|_{(t,x)}&=\tfrac{2L}{\veps}\\
			(\pr_t\hat{f}_c)|_{(\a_c(t,x),x)}\;(\pr_i\a_c)|_{(t,x)}+ (\pr_i\hat{f}_c)|_{(\a_c(t,x),x)}&= 0.
		\end{align}
		In particular  $(\pr_t\a_c)(t,x), (\pr_i\a_c)(t,x) \to 0$ as $t\to -\tfrac{\veps}{2}^-$. Further, for any $(t,x)\in (-\veps, -\tfrac{\veps}{2})\times \pr U$ we can guarantee either 
		\begin{itemize}
			\item $\partial_t \hat{f}_c|_{(\a_c(t,x),x)} \,(\partial_t\a_c)|_{(t,x)}\geq 1$ and $|B_{\a_c}(t,x)|_{\omega(\a_c(t,x),x)}\leq \tfrac{1}{V(r(\a_c(t,x),x))^2} (1-\sqrt{\lambda})$ or
			\item 	$\partial_t\a_c|_{(t,x)}\geq 1$ and $|A_{\a_c}(t,x)|_{\omega(\a_c(t,x),x)}\leq 1-\sqrt{\lambda}$
		\end{itemize} 
		where $B_{\a_c}(t,x), A_{\a_c}(t,x)$ are as in \eqref{eq:defB} resp.\ \eqref{eq:defA}, but with $\a$ replaced by $\a_c$ and $\hat f$ replaced by $\hat f_c$.
\end{lemma}

\begin{proof}
	The function $\hat{f}_c^*: (t,x)\mapsto -\hat{f}_c(-t,x)$ has the same behavior near $t=0$ as $\hat{f}$ (in particular $\hat{f}_c^*(0,x)=0$ and its gradient behaves appropriately as $t\to 0^+$), so by Lemma \ref{lem:alpha-props} there exists a $\veps_0(L,\lambda,\pr U, \hat{f}^*_c)$ such that for any $0<\veps<\veps_0(L,\lambda,\pr U, \hat{f}^*_c)$ there exists a suitable $\a:(0,\veps)\times \pr U\to (0,\veps)$ corresponding to $\hat{f}_c^*$. We set $\a_c (t,x)= -\a(-2t-\tfrac{\veps}{2},x) $ for $(t,x)\in (-\veps,-\tfrac{\veps}{2})\times \pr U$ and the claims follow.
\end{proof}

\subsection{Volume and diameter estimates}\label{ssec-ap-voldiam}

We start by estimating $\vol( \widetilde{\Omega} \setminus \Omega(\r))$. First, 
\bee
 \vol( \widetilde{\Omega} \setminus \Omega(\r))=\vol(\graph(f_c))+\vol( \partial U \times(-L, 0)).
  \eee
We obtain the following estimate using the coarea formula as in Lemma \ref{lem-volO-},
\begin{align*}
	\vol ( \graph(f_c) ) &= \int_U \sqrt{1+ V^2 |\nabla^b f_c |^2}  \, \dvol_b \\
	&\leq \vol(U) + \cosh(\tfrac{\rho_0}{2})\int_{-1}^0 \H^{n-1}(f_c^{-1}(t)) \, dt \\
	&\leq (\tfrac{1}{n-1} + 2\cosh(\tfrac{\rho_0}{2})) \vol(\pr U),
\end{align*}
where we have also used the isoperimetric inequality given in Proposition \ref{prop-IsoIneq} and that by construction the level sets of $f_c$ satisfy $\H^{n-1}(f_c^{-1}(s))\leq 2\vol(\pr U)$ for $s\in (-1,0]$. 
To estimate $\vol( \partial U \times(-L, 0)   )$ we use that $g=b|_{\partial U} + V(r)^2\, ds^2$ on the cylinder, so $\vol( \partial U \times(-L, 0)   )\leq L\,\vol(\partial U)\,\cosh(\tfrac{\rho_0}{2}) $ and the desired estimate for $\vol( \widetilde{\Omega} \setminus \Omega(\r))$ follows.\\

To estimate $\diam( \widetilde{\Omega} \setminus \Omega(\r)) $ note that 
\begin{align*}
\diam( \widetilde{\Omega} \setminus \Omega(\r)))\leq  & \diam(\graph(f_c))+2 \sup_{p\in \partial U \times(-L, 0)} d_{\mathrm{Cyl}}(p,\pr U) \\
\leq & \diam(\graph(f_c))+2 \cosh(\tfrac{\rho_0}{2}) \,L 
\end{align*}
 as the metric on the cylinder part is $b|_{\partial U}+ V^2ds^2$ and, again, $V$ is bounded  above by $\cosh(\tfrac{\rho_0}{2})$ on $\overline{U}$. To estimate the diameter of $\graph(f_c)$, we go back to the explicit construction of $\hat{f}_c$ in normal exponential coordinates (see Section \ref{subsec:cappingFunc}). Fix $x\in \pr U$. 
Set $p=(E(0,x), 0)\in \Hb^{n+1}$, $q:=(E(-\tfrac{\veps_*}{2},x), -1) \in \graph(f_c)\subset \Hb^{n+1}$. Then the curve $\gamma: t\mapsto (E(t,x), \chi(t))$ for $t\in [-\tfrac{\veps_*}{2},0]$ connects $q$ to $p$ and has length 
\begin{align*}
L(\gamma) = & \int_{-\tfrac{\veps_*}{2}}^{0} \sqrt{1 + V(r(E(t,x)))^2 (\chi')^2} dt  \\
\leq & \int_{-\tfrac{\veps_*}{2}}^{0} (1+\cosh(\tfrac{\rho_0}{2})\chi'(t))dt \\
 = & \tfrac{\veps_*}{2}+\cosh(\tfrac{\rho_0}{2})\leq \tfrac{1}{2}+\cosh(\tfrac{\rho_0}{2})
\end{align*}
So  
\bee 
\max_{x\in \pr U}  d_{\graph(f_c)}\Big(  (E(0,x), 0), E(-\tfrac{\veps_*}{2},\pr U) \times \{-1 \}  \Big) 
\leq \tfrac{1}{2}+\cosh(\tfrac{\rho_0}{2}).
\eee 
Thus, using star-shapedness of $ E(-\tfrac{\veps_*}{2},\pr U)$, we readily see that 
\begin{align*}
\diam(\graph(f_c)) \leq &
\tfrac{1}{2}+\cosh(\tfrac{\rho_0}{2})+\diam \{  (r,\theta)\in \Hb^n: \theta\in \bS^{n-1}, 0\leq r\leq \rho_{-\tfrac{\veps_*}{2}}(\theta)\leq \tfrac{\rho_0}{2} \}\\
\leq & \tfrac{1}{2}+\cosh(\tfrac{\rho_0}{2})+\rho_0
\end{align*} 
establishing the desired diameter bound.\\

Finally, to show that for $L> D+\tfrac{1}{2}\sinh(\rho_0)\,\pi \sqrt{1+\gamma^2}$,  
$(\Omega(\rho), \di_{\Omega(\rho)})$  embeds in a distance preserving way into $(\widetilde \Omega, d_{\widetilde \Omega})$
take any two points in $p,q \in \Omega(\rho)$ and a piece-wise Lipschitz curve $ c:[0,1] \to \widetilde \Omega$  joining $p$ to $q$. 
Assume the curve exits $\Omega(\rho)$ at $c(s_0)\in \pr U$  and reenters $\Omega(\rho)$ at $c(s_1)\in \pr U$, if it does not enter $\graph(f_c)$ for any $s\in [s_0,s_1]$, i.e., it remains in the cylinder, then replacing the segment $c|_{[s_0,s_1]}$ by a curve contained in $\partial U\subset \Omega(\rho)$ connecting $p_0=c(s_0)$  to $p_1=c(s_1)$
will clearly yield a shorter curve. If, on the other hand, $c$ enters $\graph(f_c)$ at any $s\in [s_0,s_1]$, then $L(\gamma|_{[s_0,s_1]})>2\cosh(\rho_{min}) L\geq 2L $. We will now construct a curve connecting $c(s_0)$ to $c(s_1)$ that is entirely contained in $\Omega(\r)$ which will have length at most $2D+\sinh(\rho_0)\,\pi \sqrt{1+\gamma^2}$ (with $D$ from Definition \ref{def-Gn2}), i.e., it will be shorter than the original curve if $L> D+\tfrac{1}{2}\sinh(\rho_0)\,\pi \sqrt{1+\gamma^2}$ (cf.\ the proof of the diameter estimates for $\Omega_j(\rho)$ in Theorem \ref{thm-IFbounds}): 
 First, follow the graph of $f$ radially outward from $c(s_0)$ (resp. $c(s_1)$) until hitting $\partial \Omega(\rho_0)$ in a point $p_0$ (resp. $p_1$). Then by Definition \ref{def-Gn2} the length of each of these segments is bounded by the depth $D$. Second, connect $p_0$ and $p_1$ along an arc in $\pr \Omega(\rho_0)$. Since $V^2|\grad^b f|^2 \leq \g^2$ in the region $r\geq \tfrac{\rho_0}{2}$ by Definition \ref{def-Gn2} the length of this arc is bounded by $\sinh(\rho_0)\,\pi \sqrt{1+\gamma^2}$.

\small{
\bibliographystyle{amsalpha}
\bibliography{IFPMTAH}
}

\end{document}